\documentclass[11pt,a4paper]{article}
\usepackage{graphicx}
\usepackage{verbatim}
\usepackage{amsthm}

\newtheorem{theorem}{Theorem}[section]
\newtheorem{lemma}[theorem]{Lemma}
\newtheorem{proposition}[theorem]{Proposition}
\newtheorem{corollary}[theorem]{Corollary}
\theoremstyle{remark}
\newtheorem{remark}[theorem]{Remark}
\theoremstyle{definition}
\newtheorem{definition}[theorem]{Definition}
\theoremstyle{definition}
\newtheorem{example}[theorem]{Example}

\title{The Glimm space of the minimal tensor product of C$^{\ast}$-algebras}

\author{David McConnell\thanks{This work was supported
by the Science Foundation Ireland under grant 11/RFP/MTH3187.}\\
\textit{School of Mathematics, Trinity College, Dublin 2, Ireland} \\ \textit{mcconnd@tcd.ie} 
}

\usepackage[all]{xy}
\usepackage{amssymb}
\usepackage{amsmath}
\usepackage{fullpage}

\begin{document}
\maketitle
\begin{abstract}
We show that for C$^{\ast}$-algebras $A$ and $B$, there is a natural open bijection from $\mathrm{Glimm}(A) \times \mathrm{Glimm} (B)$ to $\mathrm{Glimm}(A \otimes_{\alpha} B )$ (where $A \otimes_{\alpha} B$ denotes the minimal C$^{\ast}$-tensor product), and identify a large class of C$^{\ast}$-algebras $A$ for which the map  is continuous for arbitrary $B$. As a consequence we determine the structure space of the centre of the multiplier algebra $ZM(A \otimes_{\alpha} B )$ in terms of $\mathrm{Glimm}(A)$ and $\mathrm{Glimm} (B)$, and give necessary and sufficient conditions for  the inclusion $ZM(A) \otimes ZM(B) \subseteq ZM(A \otimes_{\alpha} B)$ to be surjective. Further we show that when the Glimm spaces are considered as sets of ideals, the map $(G,H) \mapsto G \otimes_{\alpha} B + A \otimes_{\alpha} H$ implements the above bijection, extending a result of Kaniuth from~\cite{kaniuth} by eliminating the assumption of property (F).
\end{abstract}

The focus of our work is the relationship between a $C^*$-algebra $A$
and
its collection of primitive ideals $\mathrm{Prim}(A)$, a topological
space in the hull kernel topology, the associated complete
regularisation space
$\mathrm{Glimm}(A)$ of $\mathrm{Prim}(A)$, and in particular
how
$\mathrm{Glimm}(A \otimes_\alpha B)$ relates to $\mathrm{Glimm}(A)$
and
$\mathrm{Glimm}(B)$. Here $A \otimes_\alpha B$ is the usual (minimal)
tensor product of two $C^*$-algebras $A$ and $B$ and our work is
motivated by a desire to extend earlier work such as~\cite{kaniuth},~\cite{arch_cb},~\cite{lazar_tensor}.

The problem of representing a C$^{\ast}$-algebra as the section algebra of a bundle over a suitable base space may be viewed as that of finding a non-commutative analogue of the Gelfand-Naimark Theorem.  Significant contributions to the theory were made by Fell~\cite{fell}, Tomiyama~\cite{tomiyama2}, Dauns and Hofmann~\cite{dauns_hofmann}, Lee~\cite{lee} and others.  The topological space $\mathrm{Prim}(A)$ arises in this context as a natural choice for the base space.  However, a major obstacle to this approach is that $\mathrm{Prim}(A)$ is in general not sufficiently well-behaved as a topological space from the point of view of bundle theory.     

We define the topological space $\mathrm{Glimm}(A)$ as the complete regularisation of $\mathrm{Prim}(A)$ and denote by $\tau_{cr}$ the (completely regular) topology on this space induced by the continuous functions on $\mathrm{Prim}(A)$.  As a set, this space is the quotient of $\mathrm{Prim}(A)$ modulo the equivalence relation $\approx$ of inseparability by continuous functions on $\mathrm{Prim}(A)$.  By the \emph{Glimm ideals} of $A$ we mean the ideals given by taking the intersection of the primitive ideals in each $\approx$-equivalence class, see \S1. 

 In~\cite{dauns_hofmann}, Dauns and Hofmann showed that any C$^{\ast}$-algebra $A$ may be represented as the section algebra of an upper semicontinuous C$^{\ast}$-bundle over $\mathrm{Glimm}(A)$ if this space is locally compact, or over its Stone-\v{C}ech compactification otherwise.  Under this representation the fibre algebras are given by the Glimm quotients of $A$.  Thus in the case of the minimal tensor product of C$^{\ast}$-algebras $A$ and $B$, a natural question that arises is to determine $\mathrm{Glimm}(A \otimes_{\alpha} B )$ in terms of $\mathrm{Glimm}(A)$ and $\mathrm{Glimm}(B)$, both topologically and as a collection of ideals of $A \otimes_{\alpha} B$.  A related problem (over more general base spaces) was studied by Kirchberg and Wassermann in~\cite{kirch_wass}, and later by Archbold in~\cite{arch_cb}, by considering the fibrewise tensor product of the corresponding bundles of $A$ and $B$.

We will denote by $\mathrm{Id}'(A)$ the set of all proper norm-closed two sided ideals of $A$.  By $\mathrm{Fac}(A)$ we mean the space of kernels of factor representations of $A$, which is a topological space in the hull-kernel topology.  

There is a natural embedding of $\mathrm{Id}'(A) \times \mathrm{Id}'(B)$ into $\mathrm{Id}'(A \otimes_{\alpha} B )$ sending $(I,J) \mapsto \mathrm{ker}(q_I \otimes q_J)$, where $q_I$ and $q_J$ are the quotient maps.  The restrictions of this map to the spaces of primitive and factorial ideals are known to be homeomorphisms onto dense subspaces of $\mathrm{Prim}(A \otimes_{\alpha} B)$ and $\mathrm{Fac}(A \otimes_{\alpha} B)$ respectively, see~\cite{wulfsohn},~\cite{lazar_tensor}.  Recently A.J. Lazar has shown in~\cite{lazar_tensor} that any continuous map $f: \mathrm{Prim}(A) \times \mathrm{Prim} (B) \rightarrow Y$, where $Y$ is a $T_1$ space has a continuous 'extension' to $\mathrm{Prim}(A \otimes_{\alpha} B)$, where we identify $\mathrm{Prim}(A) \times \mathrm{Prim}(B)$ with its image under the above embedding.

We begin in \S1 with considerations of the complete regularisation of a product $X \times Y$ of topological spaces, gathering together and extending results on this theme from the literature. Central to this is the theory of \emph{w}-compact spaces, introduced by Ishii in~\cite{ishii}.  We establish (Proposition~\ref{p:wcpct}) conditions on a C$^{\ast}$-algebra $A$ that ensure that the complete regularisation of $\mathrm{Prim}(A) \times \mathrm{Prim}(B)$ is homeomorphic to the product space $\mathrm{Glimm}(A) \times \mathrm{Glimm}(B)$ for any C$^{\ast}$-algebra $B$.  In the presence of a countable approximate unit for $A$, a necessary and sufficient condition for this to occur is that $\mathrm{Glimm}(A)$ be locally compact.  In the  general case, sufficient conditions include compactness of $\mathrm{Prim}(A)$ (e.g. if $A$ is unital), or that the complete regularisation map of $\mathrm{Prim}(A)$ is open (e.g. if $A$ is quasi-standard, see~\cite{arch_som_qs}).   Local compactness of $\mathrm{Glimm}(A)$ is always a necessary condition.

Using the extension result of Lazar together with the universal property of the  complete regularisation of a topological space described in \S1, we show (Theorem~\ref{t:homeo}) that as a topological space $\mathrm{Glimm}(A\otimes_\alpha B)$ is the same as (or can be identified in a natural way with) the complete regularisation of $\mathrm{Prim}(A) \times \mathrm{Prim}(B)$. We investigate conditions under which the latter coincides with the cartesian product space $\mathrm{Glimm}(A) \times \mathrm{Glimm}(B)$, while showing that the underlying sets always agree.  In Corollary~\ref{c:wcpct2} we give some rather general conditions on $A$ or on $B$ for this coincidence but show an example in \S6 where it fails.

A well-known Corollary of the Dauns-Hofmann Theorem is the existence of a $\ast$-isomorphism identifying the centre $ZM(A)$ of the multiplier algebra of $A$ with the C$^{\ast}$-algebra of bounded continuous functions on $\mathrm{Glimm}(A)$ (equivalently, on $\mathrm{Prim}(A)$).  We show in Theorem~\ref{t:zma} that for any C$^{\ast}$-algebras $A$ and $B$, $ZM(A \otimes_{\alpha} B )$ can be identified with the bounded continuous functions on (the complete regularisation of) $\mathrm{Prim}(A) \times \mathrm{Prim}(B)$, and give necessary and sufficient conditions (Theorem~\ref{p:zma}) for $ZM(A) \otimes ZM(B) = ZM(A \otimes_{\alpha} B )$.

In \S4, we determine the set of Glimm ideals of $A \otimes_{\alpha} B$ in terms of the Glimm ideals of $A$ and $B$. In order to do this, we use an alternative construction of $\mathrm{Glimm}(A)$ based on the complete regularisation of $\mathrm{Fac}(A)$ (rather than $\mathrm{Prim}(A)$) first considered by Kaniuth in~\cite{kaniuth}.  The reason for this is the fact that there exists a continuous surjection $\mathrm{Fac}(A \otimes_{\alpha} B) \rightarrow \mathrm{Fac}(A) \times \mathrm{Fac}(B)$, while it is not known if the restriction of this map to $\mathrm{Prim}(A \otimes_{\alpha} B)$ has range $\mathrm{Prim}(A) \times \mathrm{Prim}(B)$.

We show in Theorem~\ref{t:homeo2} that the map $\mathrm{Glimm}(A) \times \mathrm{Glimm}(B) \rightarrow \mathrm{Glimm}(A \otimes_{\alpha} B )$ sending $(G,H) \mapsto G \otimes_{\alpha} B + A \otimes_{\alpha} H$ determines the homeomorphism of these spaces when considered as sets of ideals.  This extends Kaniuth's result~\cite[Theorem 2.3]{kaniuth}, which was proved under the assumption that $A \otimes_{\alpha} B$ satisfies Tomiyama's property (F), defined below.

In \S5 we consider the problem of determining conditions for which the canonical upper semicontinuous bundle representation of $A \otimes_{\alpha} B$ over $\mathrm{Glimm}(A \otimes_{\alpha} B )$ of~\cite{dauns_hofmann} is in fact continuous (i.e. $A \otimes_{\alpha} B$ defines a maximal full algebra of operator fields in the sense of Fell~\cite{fell}).  The main result of this section (Theorem~\ref{t:rho_alpha}) shows that, under the assumption that $\mathrm{ker} (q_G \otimes q_H ) = G \otimes_{\alpha} B + A \otimes_{\alpha} H$ for all pairs of Glimm ideals $(G,H)$ of $A$ and $B$, this representation of $A \otimes_{\alpha} B$ is continuous precisely when the corresponding bundle representations of $A$ and $B$ (over $\mathrm{Glimm}(A)$ and $\mathrm{Glimm}(B)$ respectively) are continuous.  We also show that, under a different assumption that does not require that these ideals be equal, continuity of $A$ and $B$ is a necessary condition for continuity of $A \otimes_{\alpha} B$ (Proposition~\ref{p:proj}).

Let $\alpha$ and $\beta$ be states of $A$ and $B$ respectively.  Then the product state $\alpha \otimes \beta$ of $A \otimes_{\alpha} B$ is defined via $(\alpha \otimes \beta)(a \otimes b ) = \alpha (a) \beta (b)$ on elementary tensors $a \otimes b$ and extended to $A \otimes_{\alpha} B$ by linearity and continuity.  If $I,J \in \mathrm{Id}'(A \otimes_{\alpha} B )$ with $I \not\subseteq J$, then we say that a state $\gamma $ of $ A \otimes_{\alpha} B $ separates $I$ and $J$ if $\gamma (J) = \{ 0 \}$ and there exists $c \in I \backslash J$ with $\gamma (c) = 1$.  The minimal tensor product $A \otimes_{\alpha} B$ is said to satisfy Tomiyama's property (F) if given any pair $I,J \in \mathrm{Id}'(A \otimes_{\alpha} B)$ with $I \neq J$, there is a product state of $A \otimes_{\alpha} B$ separating $I$ and $J$.  There are many equivalent characterisations of property (F), see~\cite[Proposition 5.1]{lazar_tensor} for example.

For any C$^{\ast}$-algebra $A$ we denote by $M(A)$ its multiplier algebra and by $Z(A)$ its centre.  We say that $A$ is $\sigma$-unital if it admits a countable approximate identity. For $I \in \mathrm{Id}'(A)$ we denote by $\mathrm{hull}(I) = \{ P \in \mathrm{Prim}(A) : P \supseteq I$ and by $\mathrm{hull}_f (I) = \{ M \in \mathrm{Fac}(A) : M \supseteq I \}$.  For a collection $\mathcal{S} \subseteq \mathrm{Id}' (A)$ we let $k(S) = \cap \{ I : I \in \mathcal{S} \}$, the \emph{kernel} of $\mathcal{S}$.

For any topological space $X$ we will denote by $C(X)$ (resp. $C^b(X)$) the $\ast$-algebra of continuous (resp. bounded continuous) complex valued functions on $X$.  If $ \alpha : X \rightarrow Y$ is a continuous map between topological spaces we will denote by $\alpha^{\ast}: C (Y) \rightarrow C(X)$ the unique $\ast$-homomorphism given by $\alpha^{\ast} (f) = f \circ \alpha$ for $f \in C(Y)$.   

\section{The complete regularisation of a product of topological spaces}

Our terminology is that a topological space $X$ is said to be \emph{completely regular} (or a \emph{Tychonoff space}) if it is a Hausdorff space, and given any closed subset $F \subseteq X$ and a point $x \in X \backslash F$, there is a continuous function $f: X \rightarrow \mathbb{R}$ with $f(x)=1$ and $f(F) = \{ 0 \}$.  

For $f \in C (X)$, denote by $\mathrm{coz} (f) = \{ x \in X : f(x) \neq 0 \}$, the \emph{cozero set} of $f$.  Replacing $f$ with $\min ( | f |, 1)$, we may assume that any cozero set in $X$ is the cozero set of some continuous function $f: X \rightarrow [0,1]$.

Recall that for any topological space $X$ there is a canonically associated completely regular space $\rho X$, called the \emph{complete regularisation} of $X$, with the property that $C(X) \equiv C( \rho X)$ and $C^b (X) \equiv C^b (\rho X)$. Furthermore the assignment of $\rho X$ to $X$ defines a covariant functor, the \emph{Tychonoff functor}, from the category of topological spaces and continuous maps to the category of completely regular spaces.

Define an equivalence relation on $X$ as follows: for $x_1, x_2 \in X$ write $x_1 \approx x_2 $ if $f (x_1) = f (x_2)$ for all $f \in C^b (X)$. Let $\rho X = X / \approx$  and let $\rho_X : X \rightarrow \rho X$ be the quotient map.  Each $f \in C^b (X)$ defines a function $f^{\rho}$ on $\rho X$ by setting $f^{\rho}([x])=f(x)$, where $[x]$ denotes the $\approx$-equivalence class of $x$.  Denote by $\tau_{cr}$ the weak topology on $\rho X$ induced by the functions $\{ f^{\rho} : f \in C^b (X) \}$.  

It is shown in~\cite[Theorem 3.9]{gill_jer} that the space $\rho X$ constructed in this way is completely regular, and that the collection $\{ \mathrm{coz} (f^{\rho}) : f \in C^b (X) \}$ forms a basis for the topology $\tau_{cr}$. Moreover, the map $g \mapsto g \circ \rho_X$ is an isomorphism of $C^b ( \rho X)$ onto $C^b (X)$. 

An alternative construction of $\rho X$ is as follows: let $I$ denote the closed unit interval and $I^X$ the set of all continuous maps $f: X \rightarrow I$.  Let $P(X) = \prod \{ I_f : f \in I^X \}$, where $I_f = I$ for all $f \in I^X$, and let $\tau: X \rightarrow P(X)$ be defined via $\tau (x) = (f(x))_{f \in I^X}$.  

Defining $I^{\rho X}, P( \rho X)$ and $\tau': \rho X \rightarrow P(\rho X)$ analogously, it is clear that $I^{\rho X} = \{ f^{\rho} : f \in I^X \}$ and hence $P( \rho X) = P(X)$. The definition of $f^{\rho}$, where $f \in I^X$, ensures that $\tau' \circ \rho_X = \tau$, hence $\tau(X)$ is equal to $\tau' ( \rho X)$.
On the other hand, it is well-known that $\tau'$ is a homeomorphism onto its image~\cite[Lemma 1.5]{walker}, so that $\tau (X)$ is homeomorphic to $\rho X$. 

The topological space $(\rho X, \tau_{cr} )$, or more precisely the triple $(\rho X , \tau_{cr}, \rho_X)$ constructed in this way is the complete regularisation of $X$.

Now let $X$ and $Y$ be topological spaces and $\phi: X \rightarrow Y$ a continuous map.  Then setting $\phi^{\rho} ( \rho_X (x))  = ( \rho_Y \circ \phi ) (x)$ gives  a map $\phi^{\rho} : \rho X \rightarrow \rho Y$ such that the diagram
\[
\xymatrix{ X \ar@{->}^{\phi}[r] \ar_{\rho_X }[d] & Y \ar^{\rho_Y}[d]  \\
\rho X \ar_{\phi^{\rho}}[r] & \rho Y 
}
\]
commutes.  To see that $\phi^{\rho}$ is continuous, let $U = \mathrm{coz} (f)$ be a cozero set in $\rho Y$.  Then $(\rho_Y \circ \phi)^{-1} (U)$ is precisely the cozero set in $X$ of the continuous function $f \circ \rho_Y \circ \phi $.  It follows that $(\phi^{\rho})^{-1} (U) = \mathrm{coz} ( f \circ \rho_Y \circ \phi )^{\rho}$ is open in $\rho X$ by the definition of $\tau_{cr}$.

Thus the assignment of $\rho X$ to $X$ defines a covariant functor from the category of topological spaces to the subcategory of completely regular  spaces, called the \emph{Tychonoff functor}.  The term Tychonoff functor was first used by Morita in~\cite[p. 32]{morita}.

\begin{lemma}
\label{l:prodcr}
Let $X$ and $Y$ be topological spaces.  Then there is an open bijection
\[
\rho (X) \times \rho (Y) \rightarrow \rho (X \times Y ) 
\]
sending $(\rho_X (x) , \rho_Y (y) ) \mapsto \rho_{X \times Y} (x,y)$.
\end{lemma}
\begin{proof}
We first show that $(X / \approx ) \times (Y / \approx ) = (X \times Y ) / \approx$ as sets; specifically that $(x_1,y_1) \approx (x_2,y_2)$ if and only if $x_1 \approx x_2 $ and $y_1 \approx y_2$.  Indeed, if $(x_1,y_1) \approx (x_2,y_2)$ and $f \in C^b (X)$ then $f \circ \pi_X \in C^b (X \times Y)$, hence $f(x_1) = f \circ \pi_X (x_1,y_1) = f \circ \pi_X (x_2,y_2) = f(x_2)$, and $x_1 \approx x_2$.

On the other hand if both $x_1 \approx x_2$ and $y_1 \approx y_2$, take $g \in C^b (X \times Y)$; then $g(x_1,y_1)=g(x_2,y_1)=g(x_2,y_2)$.  It follows that the mapping $(\rho_X (x) , \rho_Y (y) ) \mapsto \rho_{X \times Y} (x,y)$ is a well-defined bijection.

In order to show that the above map is open we use the fact that in a completely regular space, the cozero sets of continuous functions form a base for the topology~\cite[3.4]{gill_jer}. Consider a basic open set $\mathrm{coz} (f^{\rho}) \times \mathrm{coz} (g^{\rho}) $ in $\rho X \times \rho Y$, where $f \in C^b( X),g \in C^b (  Y)$.  Then $h(x,y)=(f^{\rho} \circ \rho_X)(x)(g^{\rho} \circ \rho_Y)(y)$ defines an element of $C^b(X \times Y)$, and hence gives $h^{\rho} \in C^b ( \rho ( X \times  Y ) )$ with $h^{\rho} \circ \rho_{X \times Y} = h$.  Then $\mathrm{coz} (h^{\rho}) = \mathrm{coz} (f^{\rho}) \times \mathrm{coz} (g^{\rho})$, so that $\mathrm{coz} (f^{\rho}) \times \mathrm{coz} (g^{\rho})$ is an open subset of $\rho ( X \times Y)$.
\end{proof}

In view of Lemma~\ref{l:prodcr}, for any topological spaces $X$ and $Y$ we identify $\rho(X) \times \rho (Y)$ and $\rho( X \times Y)$ as sets.  This canonical map is not a homeomorphism in general, however (see Example~\ref{e:notlocallycompact}). Thus in what follows, we will denote by $\rho X \times \rho Y$ this product space with the (possibly weaker) product topology $\tau_p$, and by $\rho( X \times Y)$ the space with the topology $\tau_{cr}$ induced by the functions in $C^b (X \times Y)$. 

The following result was obtained originally in~\cite{hosh_mor}:
\begin{theorem}\emph{(Hoshina and Morita)}~\cite[Theorem 2.4]{mor_ish}

Let $X$ and $Y$ be topological spaces. The following are equivalent:
\begin{enumerate}
\item[(i)] $\rho X \times \rho Y = \rho ( X \times Y)$,
\item[(ii)] For any cozero set $G$ of $X \times Y$ and any point $(x,y) \in G$ there are cozero sets $U$ and $V$ of $X$ and $Y$ respectively with $(x,y) \in U \times V \subseteq G$.
\end{enumerate}
\end{theorem}

\begin{definition}
Let $(X,\mathcal{T})$ be a topological space.  For a subspace $Y \subseteq X$ denote by $\tau_Y$ the topology on $Y$ generated by $\{ \mathrm{coz} (f) : f \in C(Y) \}$. \end{definition}  

For a topological space $(X, \mathcal{T})$ and a subset $U \subseteq X$ we denote by $\mathcal{T} \restriction_U$  the subspace topology on $U$ inherited from $\mathcal{T}$.  For two topologies $\mathcal{T}_1,\mathcal{T}_2$ on $X$ we write $\mathcal{T}_1 \leq \mathcal{T}_2$ to say that $\mathcal{T}_1$ is weaker than $\mathcal{T}_2$.

For any subset $U \subseteq X$, we denote by $\overline{U}$ the $\mathcal{T}$-closure of $U$ and by $\mathrm{cl}_{\tau_X} (U)$ the $\tau_X$-closure of $U$.
The following lemma establishes some basic properties of the $\tau$-topologies on subspaces of a topological space $(X, \mathcal{T})$:

\begin{lemma}
\label{l:tau}
Let $(X, \mathcal{T})$ be a topological space and let $U \subseteq X$. Then:
\begin{enumerate}
\item[(i)] $\tau_U \leq \mathcal{T} \restriction_U$,
\item[(ii)] $U$ is $\tau_X$ open if and only if $U = \rho_X^{-1} (W)$ for some open subset $W \subseteq \rho X$,
\item[(iii)] If $U$ is $\tau_X$-open then $U$ is saturated with respect to the relation $\approx$ on $X$ (hence $X \backslash U$ is saturated also),
\item[(iv)] $\mathrm{cl}_{\tau_X} (U) = \rho_X^{-1} \left( \overline{\rho_X (U)}\right)$, where $\overline{\rho_X (U)}$ is the $\tau_{cr}$-closure of $\rho_X (U)$ in $\rho X$,
\item[(v)] If $U \subseteq V \subseteq X$ then $\tau_V \restriction_U \leq \tau_U$.
\end{enumerate}
\end{lemma}
\begin{proof}
(i) Every basic $\tau_U$-open set is of the form $\{ x \in U : f(x) \neq 0 \}$ with $f:U \rightarrow [0,1]$ continuous, hence is open in $\mathcal{T} \restriction_U$.

(ii) Note that for every continuous $f: X \rightarrow [0,1]$, $\mathrm{coz}(f) = \rho_X^{-1} ( \mathrm{coz}(f^{\rho}) )$ by the construction of $f^{\rho}$.  Hence $U \subseteq X$ is $\tau_X$-open if  and only if there exist continuous functions $f_i : X \rightarrow [0,1]$ for all $i$ in some index set $I$ such that
\[
U = \bigcup_{i \in I } \mathrm{coz} (f_i ) = \bigcup_{i \in I} \rho_X^{-1} ( \mathrm{coz} (f_i^{\rho} ) = \rho_X^{-1} \left( \bigcup_{i \in I } \mathrm{coz} (f_i^{\rho} ) \right).
\] 
Since the $\tau_{cr}$-open subsets of $\rho X$ are unions of cozero sets, the conclusion follows.

(iii) Suppose $U$ is $\tau_X$-open and $x \in U$.  Then there is a cozero set neighbourhood $\mathrm{coz}(f)$ of $x$ contained in $U$, where $f: X \rightarrow [0,1]$ is continuous.  Thus for any $y \in X \backslash U$, $f(y)=0$ and $f(x) \neq 0$.  Hence $x \not\approx y$ for any such $y$, and so $[x] \subseteq U$.

(iv) By (ii) $\rho_X^{-1} \left( \overline{\rho_X (U)} \right)$ is $\tau_X$ closed.  Now suppose $F \subseteq X$ is $\tau_X$ closed and $U \subseteq F$.  Then $\rho_X (F)$ is closed in $\rho X$ by (ii), and contains $\rho_X (U)$, hence contains $\overline{\rho_X (U)}$.  By (iii), this gives $F = \rho_X^{-1} \left( \rho_X (F) \right) \supseteq \rho_X^{-1} \left( \overline{\rho_X (U)} \right)$, as required.

(v) If $U \subseteq V \subseteq X$, then every $f \in C(V)$ has $f \restriction_U \in C(U)$.  Hence $\mathrm{coz} (f) \cap U = \mathrm{coz} (f \restriction_U )$ is a cozero set of $U$, so that the subspace topology $\tau_V \restriction_U$ is weaker than $\tau_U$. 
\end{proof}
\begin{definition} 
A topological space $(X, \mathcal{T} )$ is said to be \emph{w-compact} if given any $\mathcal{T}$-open covering $\{ U_{\alpha} \}_{\alpha \in A}$ of $X$, then there exist $\alpha_1 , \ldots \alpha_n \in A$ such that $X = \mathrm{cl}_{\tau_X} ( U_{\alpha_1} \cup \ldots \cup U_{\alpha_n})$.  
\end{definition}
It is shown in~\cite[Proposition 3.3]{mor_ish} that $X$ is w-compact if and only if any family $\{ Q_{\alpha} \}$ of $\tau_X$-open subsets of $X$ with the finite intersection property has $ \bigcap \overline{ Q_{\alpha } } \neq \emptyset$.

The class of w-compact spaces was introduced by Ishii in~\cite{ishii} to characterise the topological spaces $X$ for which $\rho ( X \times Y) = \rho X \times \rho Y$, for any topological space $Y$:

\begin{theorem}~\cite[Theorem 4.1]{mor_ish}
\label{t:ishii}
For a topological space $X$ the following are equivalent:
\begin{enumerate}
\item[(i)]  $\rho ( X \times Y ) = \rho X \times \rho Y$  for any space $Y$,
\item[(ii)] For each  $x \in X$ there is a cozero set neighbourhood $U$ of $x$ such that $\overline{U}$ is w-compact.
\end{enumerate}
\end{theorem}

We will show in Proposition~\ref{p:wcpct} that condition (ii) of Theorem~\ref{t:ishii} is satisfied by $\mathrm{Prim}(A)$ for a large class of C$^{\ast}$-algebras $A$.  The following Lemma gives a sufficient condition for a point $x$ in a general topological space $X$ to have a cozero set neighbourhood with w-compact closure.

\begin{lemma}
\label{l:wcpct}
Let $(X, \mathcal{T} )$ be a topological space and suppose that $\rho_X (x) \in \rho X$ has a compact neighbourhood $K$ such that there is a compact $C \subseteq X$ with $\rho_X (C ) = K$. Then $x$ has a cozero set neighbourhood $U$ in $X$ with $\overline{U}$ w-compact.
\end{lemma}
\begin{proof}
Choose $f \in C ( \rho X )$ with $f ( \rho_X (x) ) = 1 $ and $f ( \rho X \backslash \mathrm{int} K ) = \{ 0 \}$.  Let $U = \rho_X^{-1} ( \mathrm{coz} (f) ) = \mathrm{coz} ( f \circ \rho_X )$, a cozero set neighbourhood of $x$ in $X$.  We claim that $U \cap C$ is $\tau_U$-dense in $U$.

Let $V$ be a cozero set of $U$, then $V$ is also a cozero set of $X$ by~\cite[Lemma 3.9]{mor_ish}. Choose $g \in C( \rho X)$ such that $V = \mathrm{coz} ( g \circ \rho_X )$. Note that for any $v \in \rho_X(V) = \mathrm{coz} (g)$ there is $y \in C$ such that $\rho_X (y)=v$, hence $g \circ \rho_X (y) = g(v) \neq 0$. So $V \cap C$ is non-empty, and since every $\tau_U$ open subset of $U$ is a union of cozero sets such as $V$, $U \cap C$ is $\tau_U$-dense.

If $Q \subseteq \overline{U}$ is a nonempty $\tau_{\overline{U}}$-open subset, then it is relatively open (in the subspace topology $\mathcal{T} \restriction_{\overline{U}}$) by Lemma~\ref{l:tau} (i).  In particular,  $Q \cap U$ is nonempty, and moreover is $\tau_U$ open by Lemma~\ref{l:tau} (v).

Take a collection $\{Q_{\alpha} \}$ of $\tau_{\overline{U}}$-open subsets of $\overline{U}$ with the finite intersection property. Then for every finite subcollection $\{ Q_{\alpha_j} \}_{j=1}^n$, the intersection $\bigcap_{j=1}^n (Q_{\alpha_j} \cap U) = \left( \bigcap_{j=1}^n Q_{\alpha_j} \right) \cap U$ is  nonempty. It follows that $\{ Q_{\alpha} \cap U \}$ is a collection of  $\tau_U$-open subsets of $U$ with the finite intersection property.  Since $U \cap C$ is $\tau_U$-dense,
\[ \bigcap_{j=1}^n ( Q_{\alpha_j} \cap U \cap C ) = \left ( \bigcap_{j=1}^n Q_{\alpha_j} \cap U \right) \cap C \]
is nonempty for every such subcollection.  Thus $\{ Q_{\alpha} \cap U \cap C \}$ is a collection of subsets of $C$ with the finite intersection property.  Since $C$ is compact, $\bigcap \overline{ ( Q_{\alpha} \cap U \cap C )} \cap C \neq \emptyset$.  As 
\[ \overline{( Q_{\alpha} \cap U \cap C ) } \cap C  \subseteq  \overline{Q_{\alpha}}\]
for each $\alpha$, this implies that $\bigcap \overline{Q_{\alpha}} \neq \emptyset $. Hence $\overline{U}$ is w-compact.
\end{proof}

\begin{lemma}
\label{l:wcpct2}
Let $X$ be a topological space and suppose that every $x \in X$ has a cozero set neighbourhood with w-compact closure.  Then $\rho X$ is locally compact.
\end{lemma}
\begin{proof}
If $A \subseteq X$ is w-compact, then $\rho_X (A)$ is w-compact by~\cite[Proposition 3.10]{mor_ish}.  But then since $\rho_X (A)$ is completely regular it is homeomorphic to its complete regularisation $\rho \left( \rho_X (A) \right)$, hence is compact by~\cite[Proposition 3.4]{mor_ish}.

For each point $x \in X$, let $U_x$ be a cozero set neighbourhood of $x$ with $\overline{U_x}$ w-compact.  Then $\rho_X ( \overline{U_x} )$ is compact, and is a neighbourhood of $\rho_{X} (x)$ since $\rho_X (U_x)$ is open by Lemma~\ref{l:tau}(iv). 
\end{proof}

Now let $A$ be a C$^{\ast}$-algebra and $\mathrm{Prim}(A)$ the space of primitive ideals of $A$ with the hull-kernel topology.  We define $\mathrm{Glimm}(A)$ as the complete regularisation $\rho \mathrm{Prim}(A)$ of $\mathrm{Prim} (A)$, and denote by $\rho_A : \mathrm{Prim} (A) \rightarrow \mathrm{Glimm} (A)$ the complete regularisation map.

Let $p \in \mathrm{Glimm} (A)$ and choose $P \in \mathrm{Prim} (A)$ with $\rho_A (P)=p$.  We associate to the point $p$ the norm closed two sided ideal $G_p$ of $A$ given by
\[
G_p = \bigcap \{ Q \in \mathrm{Prim} (A) : Q \approx P \} = \bigcap \{ Q \in \mathrm{Prim} (A) : \rho_A (Q) = p \} = k \left( [P ] \right).
\]
 Note that since $[P]$ is closed in $\mathrm{Prim} (A)$ and $G_p = k ( [P])$, each equivalence class in $\mathrm{Prim} (A) / \approx$ is of the form
\[
[P] = \mathrm{hull}\left( k([P]) \right) = \mathrm{hull} (G_p),
\] 
by the definition of the hull-kernel topology.  

The collection $\{ G_p : p \in \mathrm{Glimm} (A) \}$ are known as the \emph{Glimm ideals} of $A$.  Since the assignment $p \mapsto G_p$ is injective, we will regard elements of $\mathrm{Glimm}(A)$ as either points of a topological space or as ideals of $A$, depending on the context. 
\begin{proposition}
\label{p:wcpct}
Let $A$ be a C$^{\ast}$-algebra such that one of the following conditions hold:
\begin{enumerate}
\item[(i)] $\mathrm{Prim} (A)$ is compact,
\item[(ii)] the complete regularisation map $\rho_A : \mathrm{Prim} (A) \rightarrow \mathrm{Glimm} (A)$ is open, or
\item[(iii)] $A$ is $\sigma$-unital and $\mathrm{Glimm} (A)$ is locally compact.
\end{enumerate}
Then every $P \in \mathrm{Prim} (A)$ has a cozero set neighbourhood with w-compact closure.  Hence for any C$^{\ast}$-algebra $B$, the complete regularisation $\rho( \mathrm{Prim}(A) \times \mathrm{Prim}(B) )$ of $\mathrm{Prim}(A) \times \mathrm{Prim}(B)$ is homeomorphic to the product space $\mathrm{Glimm}(A) \times \mathrm{Glimm}(B)$.

Conversely, if $\mathrm{Glimm} (A)$ is not locally compact then there is $P \in \mathrm{Prim} (A)$ that does not have a cozero set neighbourhood with w-compact closure.
\end{proposition}
\begin{proof}
Note that if (i) holds then $\mathrm{Glimm}(A)$, being the continuous image of the compact space $\mathrm{Prim}(A)$, is compact. The proposition is then immediate by Lemma~\ref{l:wcpct} with $C= \mathrm{Prim} (A)$ and $K = \mathrm{Glimm} (A)$.

In cases (ii) and (iii), take $P \in \mathrm{Prim} (A)$ with $\rho_A (P) = x$ and let $K'$ be a compact neighbourhood of $x$ in $\mathrm{Glimm} (A)$. By~\cite[Theorem 2.1 and Proposition 2.5]{lazar_glimm}, $K'$ is contained in a compact subset of  $\mathrm{Glimm} (A)$ of the form
\[
K:= \{ G \in \mathrm{Glimm} (A) : \| a + G \| \geq \alpha \} = \rho_A ( \{ P \in \mathrm{Prim} (A) : \| a+P \| \geq \alpha \} ),
\]
for some $a \in A$ and $\alpha > 0 $, and the set $ \{ P \in \mathrm{Prim} (A) : \| a + P \| \geq \alpha \}$ is compact by~\cite[Proposition 3.3.7]{dix}. Then $K$ is a compact neighbourhood of $x$, and the conclusion thus follows from Lemma~\ref{l:wcpct}.

It then follows from Theorem~\ref{t:ishii} that if any of the conditions (i) to (iii) hold, $\rho ( \mathrm{Prim}(A) \times Y ) = \rho( \mathrm{Prim}(A) ) \times \rho (Y)$ for any space $Y$.  In particular, if $B$ is a C$^{\ast}$-algebra then we have
\[
\rho \left( \mathrm{Prim}(A) \times \mathrm{Prim} (B) \right) = \rho \left( \mathrm{Prim}(A) \right) \times \rho \left( \mathrm{Prim}(B) \right) = \mathrm{Glimm}(A) \times \mathrm{Glimm}(B).
\]

On the other hand if $\mathrm{Glimm}(A)$ is not locally compact, then by Lemma~\ref{l:wcpct2} there is $P \in \mathrm{Prim}(A)$ for which no cozero set neighbourhood of $P$ has w-compact closure.

\end{proof}

\begin{remark}
Suppose that $A$ is a C$^{\ast}$-algebra such that $\mathrm{Prim}(A)$ does not satisfy condition (ii) of Theorem~\ref{t:ishii}.  Then there is a topological space $Y$ for which $\rho ( \mathrm{Prim}(A) \times Y ) \neq \mathrm{Glimm}(A) \times \rho(Y)$.  It is not immediately evident whether this space $Y$ can be chosen as $\mathrm{Prim} (B)$ for some C$^{\ast}$-algebra $B$.  Thus the partial converse in Proposition~\ref{p:wcpct} does not preclude the possibility that $\rho \left( \mathrm{Prim}(A) \times \mathrm{Prim}(B) \right) = \mathrm{Glimm}(A) \times \mathrm{Glimm}(B)$ for all C$^{\ast}$-algebras $A$ and $B$.

We will show in Example~\ref{e:notlocallycompact} however that $\rho(X \times Y ) \neq \rho (X) \times \rho (Y)$ is indeed possible when $X$ and $Y$ are primitive ideal spaces of C$^{\ast}$-algebras.  Specifically, we construct  a C$^{\ast}$-algebra $A$ for which $\rho ( \mathrm{Prim}(A) \times \mathrm{Prim}(A) ) \neq \mathrm{Glimm}(A) \times \mathrm{Glimm}(A)$
\end{remark}
\begin{remark}
\label{r:tauq}
Another natural topology on the complete regularisation $\rho X$ of a space $X$ is the quotient topology $\tau_q$ induced by the complete regularisation map $\rho_X$; that is, the strongest topology on $\rho X$ for which $\rho_X$ is continuous.  Since $\rho_X$ is continuous as a map into $(\rho X, \tau_{cr} )$, it always holds that $\tau_{cr} \leq \tau_q$.  However, there is an example due to D.W.B. Somerset of a space $X$ for which $\tau_{cr} \neq \tau_q$ on $\rho X$, and a C$^{\ast}$-algebra $A$ with $\mathrm{Prim} (A)$ homeomorphic to $X$~\cite[Appendix]{lazar_quot}. 

In the case of the primitive ideal space of a C$^{\ast}$-algebra $A$, there are many conditions known to ensure that $\tau_{cr} = \tau_q$ on $\mathrm{Glimm}(A)$.  A.J. Lazar has shown in~\cite[Theorem 2.6]{lazar_quot} that if $X$ is locally compact and $\sigma$-compact (that is, $X$ is a countable union of compact subsets) then $\tau_{cr} = \tau_q$ on $\rho X$.  In particular this holds for the space $\mathrm{Prim}(A)$ whenever $A$ is a $\sigma$-unital C$^{\ast}$-algebra, or when $\mathrm{Prim}(A)$ is compact.  The two topologies also coincide if $\rho_A$ is either  $\tau_{cr}$ or $\tau_q$-open~\cite[p. 351]{arch_som_qs}.

One particular consequence of these results is that if $A$ is a C$^{\ast}$-algebra satisfying one of the conditions (i) to (iii) of Proposition~\ref{p:wcpct}, then necessarily $\tau_{cr} = \tau_q$ on $\mathrm{Glimm}(A)$.

The topology $\tau_{cr}$ is in some sense the more natural topology on $\rho X$, since it is by definition the unique topology for which $\rho_X^{\ast} : C^b (\rho X ) \rightarrow C^b (X)$ is a $\ast$-isomorphism.  This allows us to apply the results  of Ishii from~\cite{mor_ish} on the complete regularisation of a product space. Moreover, in the case of the primitive ideal space of a C$^{\ast}$-algebra $A$, we will apply the Dauns-Hofmann identification (see \S\ref{s:zma}) of $ZM(A)$ with $C^b ( \mathrm{Glimm} (A), \tau_{cr}  )$ to determine $ZM(A \otimes_{\alpha} B )$ in terms of continuous functions on the Glimm spaces of $A$ and $B$.

\end{remark}

\section{The Glimm space of the minimal tensor product of C$^{\ast}$-algebras}

In this section we show that, as a topological space $\mathrm{Glimm}(A \otimes_{\alpha} B )$ can be naturally identified with $\mathrm{Glimm}(A) \times \mathrm{Glimm} (B)$, when the latter space is considered as the complete regularisation of $\mathrm{Prim}(A) \times \mathrm{Prim} (B)$.  We first discuss the canonical embedding of $\mathrm{Prim}(A) \times \mathrm{Prim} (B)$ in $\mathrm{Prim}(A \otimes_{\alpha} B )$.

Let $\pi: A \rightarrow A'$ and $\sigma : B \rightarrow B'$ be $\ast$-homomorphisms of C$^{\ast}$-algebras.  Then there is a unique $\ast$-homomorphism $\pi \otimes \sigma : A \otimes_{\alpha} B \rightarrow A' \otimes_{\alpha} B'$, such that $(\pi \otimes \sigma) ( a \otimes b) = \pi (a) \otimes \sigma (b)$ for all elementary tensors $a \otimes b \in A \otimes B$.  In particular let $(I,J) \in \mathrm{Id} ' (A) \times \mathrm{Id} ' (B)$ and denote by $q_I : A \rightarrow A / I, q_J: B \rightarrow B/J$ the quotient homomorphisms.  Then we have a $\ast$-homomorphism $q_I \otimes q_J : A \otimes_{\alpha} B \rightarrow (A/I) \otimes_{\alpha} (B / J)$.

We now define two natural  maps $\Phi, \Delta : \mathrm{Id}' (A) \times \mathrm{Id}' (B) \rightarrow \mathrm{Id}' ( A \otimes_{\alpha} B)$ via
\begin{eqnarray}
\label{e:phi} \Phi (I,J) &=& \mathrm{ker} ( q_I \otimes q_J ) \\
\label{e:delta} \Delta (I,J) &=& I \otimes_{\alpha} B + A \otimes_{\alpha} J.
\end{eqnarray}

 The following Proposition lists some known properties of the map $\Phi$.

\begin{proposition}
\label{p:phi}
Let $A$ and $B$ be C$^{\ast}$-algebras and $A \otimes_{\alpha} B$ their minimal C$^{\ast}$-tensor product.  Then the map $\Phi$ defined by (\ref{e:phi}) has the following properties:
\begin{enumerate}
\item[(i)]  If $I,K \in \mathrm{Id}' (A)$ and $J,L \in \mathrm{Id}' (B)$ are such that $I \supseteq K$ and $J \supseteq L$ then $\Phi (I,J) \supseteq \Phi (K,L)$ \emph{~\cite[Lemma 2.2]{lazar_tensor}},
\item[(ii)] The restriction of $\Phi$ to $\mathrm{Prim}(A) \times \mathrm{Prim} (B)$ is a homeomorphism onto its image which is dense in $\mathrm{Prim}(A \otimes_{\alpha} B )$\emph{~\cite[lemme 16]{wulfsohn}},
\item[(iii)] The restriction of $\Phi$ to $\mathrm{Fac}(A) \times \mathrm{Fac} (B)$ is a homeomorphism onto its image which is dense in $\mathrm{Fac} (A \otimes_{\alpha} B )$\emph{~\cite[Corollary 2.7]{lazar_tensor}},
\item[(iv)] For $I,J \in \mathrm{Id}'(A) \times \mathrm{Id}' (B)$, $\Phi \left( \mathrm{hull} (I) \times \mathrm{hull} (J) \right)$ is dense in $\mathrm{hull}( \Phi(I,J) )$\emph{~\cite[Corollary 2.3]{lazar_tensor}},
\item[(v)] For $I,J \in \mathrm{Id}'(A) \times \mathrm{Id}' (B)$, we have
\[
\Phi (I,J) = \bigcap \{ \Phi (P,Q): (P,Q) \in \mathrm{hull}(I) \times \mathrm{hull}(J) \}
\]
\emph{\cite[Remark 2.4]{lazar_tensor}}
\end{enumerate}
\end{proposition}

Theorem~\ref{t:homeo} below identifies the complete regularisation of $\mathrm{Prim}(A) \times \mathrm{Prim} (B)$ with that of $\mathrm{Prim}(A \otimes_{\alpha} B)$.  As discussed in the remarks proceeding Lemma~\ref{l:prodcr}, we need to take into account the appropriate topology on the former space.  Thus we will refer to  $\rho \left( \mathrm{Prim}(A) \times \mathrm{Prim}(B) \right)$ as $\left( \mathrm{Glimm}(A) \times \mathrm{Glimm}(B), \tau_{cr} \right)$, and  $\rho \left( \mathrm{Prim} (A) \right) \times \rho \left( \mathrm{Prim} (B) \right)$ as $\left( \mathrm{Glimm}(A) \times \mathrm{Glimm} (B) , \tau_p \right)$
\begin{theorem}
\label{t:homeo}
Let $A$ and $B$ be C$^{\ast}$-algebras, $A \otimes_{\alpha} B$ their minimal C$^{\ast}$-tensor product and denote by $\rho_A,\rho_B$ and $\rho_{\alpha}$ the complete regularisation maps of $\mathrm{Prim}(A),\mathrm{Prim}(B)$ and $\mathrm{Prim} (A \otimes_{\alpha}B )$ respectively.  Then there is a homeomorphism $\psi : \mathrm{Glimm} ( A \otimes_{\alpha} B ) \rightarrow ( \mathrm{Glimm} (A) \times \mathrm{Glimm} (B), \tau_{cr} )$ given by
\[
(\psi \circ \rho_{\alpha}) \left( \Phi (P,Q) \right) = \left( \rho_A (P), \rho_B (Q) \right).
\]
It follows that $\psi$ defines a continuous bijection $\mathrm{Glimm} ( A \otimes_{\alpha} B ) \rightarrow ( \mathrm{Glimm} (A) \times \mathrm{Glimm} (B), \tau_p )$.
\end{theorem}
\begin{proof}
 The map $\rho_A \times \rho_B : \mathrm{Prim} (A) \times \mathrm{Prim} (B) \rightarrow ( \mathrm{Glimm} (A) \times \mathrm{Glimm} (B) , \tau_{cr} )$ is the complete regularisation map of $\mathrm{Prim} (A) \times \mathrm{Prim} (B) $ by Lemma~\ref{l:prodcr}.  For the remainder of the proof we will consider $\mathrm{Glimm}(A) \times \mathrm{Glimm}(B)$ with this topology (from which the second assertion will follow since $\tau_p$ is weaker).

  By~\cite[Theorem 3.2]{lazar_tensor}, the map $(\rho_A \times \rho_B) \circ \Phi^{-1} : \Phi  \left( \mathrm{Prim} (A) \times \mathrm{Prim} (B) \right) \rightarrow \mathrm{Glimm} (A) \times \mathrm{Glimm} (B) $ extends uniquely to a continuous map $\overline{(\rho_A \times \rho_B)} : \mathrm{Prim} (A \otimes_{\alpha} B )\rightarrow \mathrm{Glimm} (A) \times \mathrm{Glimm} (B)$.

Since $\mathrm{Glimm} (A) \times \mathrm{Glimm} (B)$ is completely regular, $\overline{ ( \rho_A \times \rho_B ) }$  induces a continuous (surjective) map $\psi : \mathrm{Glimm} (A \otimes_{\alpha } B) \rightarrow \mathrm{Glimm} (A) \times \mathrm{Glimm} (B) $ with the property that $ \psi \circ \rho_{\alpha} = \overline{ ( \rho_A \times \rho_B )}$~\cite[Corollary 1.8]{walker}.  
\[
\xymatrix{ \mathrm{Prim} (A) \times \mathrm{Prim} (B) \ar@{^{(}->}^{\Phi}[r] \ar_{\rho_A \times \rho_B }[d] & \mathrm{Prim} (A \otimes_{\alpha} B ) \ar^{\rho_{\alpha}}[d] \ar_{\overline{\rho_A \times \rho_B}}[ld] \\
\mathrm{Glimm} (A) \times \mathrm{Glimm} (B) & \ar@{-->}^{\psi}[l] \mathrm{Glimm} ( A \otimes_{\alpha} B ) 
}
\]

To show that $\psi$ is in fact a homeomorphism, it suffices to show that the $\ast$-homomorphism $\psi^{\ast} : C^b ( \mathrm{Glimm} (A) \times \mathrm{Glimm} (B) ) \rightarrow C^b ( \mathrm{Glimm} (A \otimes_{\alpha} B ))$, $\psi^{\ast} (f) = f \circ \psi$ is surjective~\cite[Theorem 10.3 (b)]{gill_jer}.

To this end, let $f \in C^b ( \mathrm{Glimm} ( A \otimes_{\alpha} B ))$, so that $f \circ \rho_{\alpha} \in C^b \left( \mathrm{Prim} ( A \otimes_{\alpha} B ) \right)$ and hence $f \circ \rho_{\alpha} \circ \Phi \in C^b \left( \mathrm{Prim} (A) \times \mathrm{Prim} (B) \right)$.    Denote by $g \in C^b ( \mathrm{Glimm} (A) \times \mathrm{Glimm} (B) )$ the unique function such that $g \circ ( \rho_A \times \rho_B ) = f \circ \rho_{\alpha} \circ \Phi$.  Then $f  \circ \rho_{\alpha}$ and $g \circ \overline{( \rho_A \times \rho_B )}$ are both continuous extensions of $g \circ ( \rho_A \times \rho_B ) \circ \Phi^{-1}$ to $\mathrm{Prim} (A \otimes_{\alpha} B )$, hence must agree by~\cite[Theorem 3.2]{lazar_tensor}. 

Take $m \in \mathrm{Glimm} ( A \otimes_{\alpha} B)$ and $M \in \mathrm{Prim} ( A \otimes_{\alpha} B )$ such that $\rho_{\alpha} (M) = m$.  Then
\begin{eqnarray*}
\psi^{\ast}(g) (m) = (g \circ \psi ) (m) & = & (g \circ \psi \circ \rho_{\alpha} )(M) \\
& = & ( g \circ \overline{( \rho_A \times \rho_B )} )(M) \\
& = & (f \circ \rho_{\alpha} )(M) \\
& = & f (m) 
\end{eqnarray*}
It follows that $\psi^{\ast} (g) = f$, hence $\psi^{\ast}$ is surjective.
\end{proof}

Note that Theorem~\ref{t:homeo} shows that $\rho_{\alpha} \circ \Phi$ is surjective.  In particular, given any $M \in \mathrm{Prim}(A \otimes_{\alpha} B )$ there exist $(P,Q) \in \mathrm{Prim}(A) \times \mathrm{Prim}(B)$ such that $M \approx \Phi (P,Q)$.

\begin{corollary}
\label{c:wcpct2}
Let $A$ and $B$ be C$^{\ast}$-algebras such that either $A$ or $B$ satisfies one of the  conditions (i)-(iii) of Proposition~\ref{p:wcpct}.  Then $\tau_{cr} = \tau_p$ on $\mathrm{Glimm} (A) \times \mathrm{Glimm} (B)$, and hence $\mathrm{Glimm}(A \otimes_{\alpha} B)$ is homeomorphic to $ \left ( \mathrm{Glimm}(A) \times \mathrm{Glimm} (B) , \tau_p \right) $ via the map $\psi$ of Theorem~\ref{t:homeo}.
\end{corollary}
\begin{proof}
Immediate from Proposition~\ref{p:wcpct} and Theorem~\ref{t:homeo}.
\end{proof}

\section{The central multipliers of $A \otimes_{\alpha} B$}
\label{s:zma}
In this section we  apply Theorem~\ref{t:homeo} to determine the centre of the multiplier algebra of $A \otimes_{\alpha} B$ in terms of the topological space $( \mathrm{Glimm}(A) \times \mathrm{Glimm} (B) , \tau_{cr} )$. We show in Theorem~\ref{t:zma} that $ZM(A \otimes_{\alpha} B )$ is $\ast$-isomorphic to the C$^{\ast}$-algebra of continuous functions on the Stone-\v{C}ech compactification of $( \mathrm{Glimm}(A) \times \mathrm{Glimm} (B) , \tau_{cr} )$.  Further in Theorem~\ref{p:zma} we give necessary and sufficient conditions on this space for  which $ZM(A) \otimes ZM(B) = ZM(A \otimes_{\alpha} B)$. 

 The embedding of $M(A) \otimes_{\alpha} M(B) \subseteq M(A \otimes_{\alpha} B )$ is discussed in~\cite{apt}, we include a proof in Lemma~\ref{l:embed} below for completeness. It is shown in~\cite[Corollary 1]{hay_wass} that for C$^{\ast}$-algebras $C$ and $D$ we have $Z(C \otimes_{\alpha} D) = Z(C) \otimes Z(D)$ (where $Z(C) \otimes Z(D)$ is the unique C$^{\ast}$-completion of the algebraic tensor product $Z(C) \odot Z(D)$ by nuclearity).  In particular it follows that for any C$^{\ast}$-algebras $A$ and $B$ we may identify $Z( M(A) \otimes_{\alpha} M(B) ) = ZM(A) \otimes ZM(B)$.  Thus in this section we are concerned with relating the centre of the larger algebra $M(A \otimes_{\alpha} B)$  with that of $M(A) \otimes_{\alpha} M(B)$. 
 
 Suppose that $C$ is a C$^{\ast}$-algebra and $z \in M(C)$ such that $zc=cz$ for all $c \in C$, and take $m \in M(C)$.  Then
 \[
 (zm)c = z(mc) = (mc)z = m (cz ) = m (zc) = (mz)c,
 \]
and similarly $c (zm) = c (mz)$ for all $c \in C$.  Thus $ zm = mz $ and so we may identify
\begin{equation}
\label{eq:zma}
ZM(C) = \{ z \in M(C) : zc = cz \mbox{ for all } c \in C \}.
\end{equation}

Recall that an ideal $I$ of a C$^{\ast}$-algebra $C$ is said to be \emph{essential} in $C$ if given any nonzero ideal $J$ of $C$, $J \cap I \neq \{ 0 \}$.  Equivalently for any $c \in C$, $cI = Ic = \{ 0 \}$ implies $c = 0$.
\begin{lemma}
\label{l:embed}
There is a canonical embedding $\Theta : M(A) \otimes_{\alpha} M(B) \rightarrow M(A \otimes_{\alpha} B )$ such that
\[
\left( \Theta ( x \otimes y ) \right) ( a \otimes b ) = xa \otimes yb \mbox{ and } (a \otimes b ) \left( \Theta (x \otimes y ) \right) = ax \otimes by
\]
for all $a \in A, b \in B, x \in M(A), y \in M(B)$.  Moreover, $\Theta ( ZM(A) \otimes ZM(B) ) \subseteq ZM(A \otimes_{\alpha} B)$.
\end{lemma}
\begin{proof}
Clearly $M(A) \otimes_{\alpha} M(B)$ contains $A \otimes_{\alpha} B$ as a two-sided ideal.  Suppose $J$ is a nonzero ideal of $M(A) \otimes_{\alpha} M(B)$.  Then by~\cite[Proposition 4.5]{allen_sinclair_smith}, $J$ contains a nonzero elementary tensor $x \otimes y$ where $x \in M(A), y \in M(B)$.  Since $A$ is essential in $M(A)$, there is $a \in A$ with either $ax \neq 0 $ or $xa \neq 0 $.  Suppose w.l.o.g. that $xa \neq 0 $, so that  $\| (xa)^{\ast} xa \| = \| xa \|^2 \neq 0$. Setting $a' = (xa)^{\ast}$, we then have an element $a' \in A$ with $a'xa \neq 0$. Similarly there are $b, b' \in B$ with $b'yb \neq 0 $.  It follows that
\[
a'xa \otimes b'yb = (a' \otimes b')(x \otimes y )(a \otimes b )
\]
is a nonzero element of $J \cap (A \otimes_{\alpha} B )$. Hence $A \otimes_{\alpha} B$ is essential in $M(A) \otimes_{\alpha} M(B)$.

By~\cite[Proposition 3.7 (i) and (ii)]{busby}, there is a unique $\ast$-homomorphism $\Theta : M(A) \otimes_{\alpha} M(B) \rightarrow M(A \otimes_{\alpha} B )$ extending the canonical inclusion of $A \otimes_{\alpha} B$ into $M(A \otimes_{\alpha} B )$, which is injective  since $A \otimes_{\alpha} B$ is essential in $M(A) \otimes_{\alpha} M(B)$. For elementary tensors $x \otimes y \in M(A) \otimes_{\alpha} M(B)$ and $a \otimes b \in A \otimes_{\alpha} B$ we have
\[
\left( \Theta (x \otimes y ) \right) (a \otimes b ) = \Theta (x \otimes y ) \Theta (a \otimes b ) = \Theta (xa \otimes yb ) = xa \otimes yb,
\]
(since $\Theta$ is the identity on $A \otimes_{\alpha} B$), and similarly $ (a \otimes b ) \left( \Theta (x \otimes y ) \right) = ax \otimes by$.

For elementary tensors $z_1 \otimes z_2 \in ZM(A) \otimes ZM(B)$ and $a \otimes b \in A \otimes_{\alpha} B$ we have
\[
\Theta ( z_1 \otimes z_2 )(a \otimes b ) = z_1 a \otimes z_2 b = a z_1 \otimes b z_2 =  (a \otimes b ) \Theta (z_1 \otimes z_2 ),
\]
from which it follows that for any $z \in ZM(A) \otimes ZM(B)$ and $c \in A \otimes_{\alpha} B$, $\Theta (z) c = c \Theta (z)$.  Hence by~(\ref{eq:zma}) we see that $\Theta (ZM(A) \otimes ZM(B) ) \subseteq ZM(A \otimes_{\alpha} B )$.
\end{proof}

We remark that it was shown in~\cite[Theorem 3.8]{apt} that if $A$ is $\sigma$-unital and non-unital, and $B$ is infinite dimensional, then $\Theta$ is not surjective. In what follows, we will suppress mention of $\Theta$ and simply consider $M(A) \otimes M(B) \subseteq M(A \otimes_{\alpha} B )$.

The relationship between the central multipliers of a C$^{\ast}$-algebra and its Glimm space was established by Dauns and Hofmann as a corollary to their work on sectional representation for C$^{\ast}$-algebras:

\begin{corollary}~\cite[III Corollary 8.16]{dauns_hofmann}
\label{c:dh}
For any C$^{\ast}$-algebra $A$, there is a homeomorphism of $\mathrm{Prim}\left( ZM(A) \right)$ onto $\beta \mathrm{Glimm} (A)$, and hence a $\ast$-isomorphism $\mu_A :  C ( \beta \mathrm{Glimm} (A) ) \rightarrow ZM(A)$.  Moreover, $\mu_A$ satisfies
\[
\mu_A (f) a - f(p)a \in G_p, \mbox{ for all } f \in C( \beta \mathrm{Glimm}(A) ), p \in \mathrm{Glimm}(A), a \in A,
\]
where $G_p = \bigcap \{ P \in \mathrm{Prim} (A) : \rho_A (P)=p \}$ is the Glimm ideal of $A$ corresponding to $p$.
\end{corollary}

Applying this identification to $A \otimes_{\alpha} B$ together with the homeomorphism $\psi$ of Theorem~\ref{t:homeo} allows us to determine $ZM(A \otimes_{\alpha} B)$ in terms of $\mathrm{Glimm}(A)$ and $\mathrm{Glimm} (B)$:

\begin{theorem}
\label{t:zma}
Let $A$ and $B$ be C$^{\ast}$-algebras and denote by $\psi$ the homeomorphism of Theorem~\ref{t:homeo}.  For each point $p \in \mathrm{Glimm}(A \otimes_{\alpha} B )$ let $G_p$ denote the Glimm ideal of $A \otimes_{\alpha} B$ corresponding to $p$.  Then there is a canonical $\ast$-isomorphism $\Theta_{\alpha} : C( \beta ( \mathrm{Glimm}(A) \times \mathrm{Glimm} (B) , \tau_{cr} ) ) \rightarrow ZM(A \otimes_{\alpha} B )$ with the property that
\[
\Theta_{\alpha} (f) c - ( f \circ \psi ) (p) c \in G_p,
\]
for all $f \in C( \beta( \mathrm{Glimm}(A) \times \mathrm{Glimm} (B) , \tau_{cr} ) ), p \in \mathrm{Glimm} (A \otimes_{\alpha} B )$ and $c \in A \otimes_{\alpha} B$.
\end{theorem}
\begin{proof}
Since $\psi$ is a homeomorphism the induced map $\psi^{\ast}$ is a $\ast$-isomorphism.  Denote by $\mu_{\alpha}$ the $\ast$-isomorphism of Corollary~\ref{c:dh} applied to $A \otimes_{\alpha} B$, and by $\Theta_{\alpha}$ the composition of the $\ast$-isomorphisms
  
\[ 
\xymatrix{
 C \left( \beta ( \mathrm{Glimm}(A) \times \mathrm{Glimm}(B) , \tau_{cr} ) \right)  \ar^<<<<<{\psi^{\ast}}[r] & C \left( \beta \mathrm{Glimm} ( A \otimes_{\alpha} B ) \right) \ar^<<<<{\mu_{\alpha}}[r] & ZM(A \otimes_{\alpha} B).
}
\]
Then $\Theta_{\alpha}$ clearly has the required properties since $\mu_{\alpha}$ does.
\end{proof}

On the other hand, applying the identification of Corollary~\ref{c:dh} to $A$ and $B$ separately gives $\ast$-isomorphisms
\[
\xymatrix{
C \left( \beta \mathrm{Glimm}(A) \times \beta \mathrm{Glimm} (B) \right) \ar^{\nu}[r] & C \left( \beta \mathrm{Glimm} (A) \right) \otimes C \left( \beta \mathrm{Glimm} (B) \right) \ar^<<<<<{\mu_A \otimes \mu_B}[r] & ZM(A) \otimes ZM(B),
}
\]
where $\nu$ is the canonical identification satisfying $\nu^{-1} ( f \otimes g ) (x,y) = f(x)g(y)$ for all elementary tensors $f \otimes g$ and $(x,y) \in \beta \mathrm{Glimm}(A) \times \beta \mathrm{Glimm} (B)$.

Let $X$ and $Y$ be completely regular spaces.  Then since the product $\beta X \times \beta Y$ is compact, the universal property of the Stone-\v{C}ech compactification~\cite[Theorem 6.5 (I)]{gill_jer} ensures that the inclusion $\iota: X \times Y \rightarrow \beta X \times \beta Y$ has a continuous extension to $\iota^{\beta} : \beta (X \times Y ) \rightarrow \beta X \times \beta Y$.  Moreover, since $\iota$ has dense range, compactness of $ \beta (X \times Y )$ implies that $\iota^{\beta}$ is necessarily surjective.

Considering $\mathrm{Glimm}(A) \subseteq \beta \mathrm{Glimm}(A)$ and $\mathrm{Glimm}(B) \subseteq \beta \mathrm{Glimm}(B)$, Lemma~\ref{l:prodcr} gives a continuous map  
\[
\phi: \left( \mathrm{Glimm}(A) \times \mathrm{Glimm} (B) , \tau_{cr} \right) \rightarrow (\mathrm{Glimm}(A) \times \mathrm{Glimm}(B), \tau_p) \subseteq \beta \mathrm{Glimm} (A) \times \beta \mathrm{Glimm} (B),
\]
and $\phi$ is a homeomorphism onto its range if and only if $\tau_p = \tau_{cr}$ on $\mathrm{Glimm}(A) \times \mathrm{Glimm} (B)$.  Again by the universal property of the Stone-Cech compactification, $\phi$ extends to a continuous surjection
\[
\phi^{\beta} : \beta \left( \mathrm{Glimm}(A) \times \mathrm{Glimm} (B) , \tau_{cr} \right) \rightarrow  \beta \mathrm{Glimm} (A) \times \beta \mathrm{Glimm} (B).
\]
Dual to this map is an injective $\ast$-homomorphism 
\[
(\phi^{\beta})^{\ast} : C \left( \beta \mathrm{Glimm} (A) \times \beta \mathrm{Glimm} (B) \right) \rightarrow C \left( \beta \left( \mathrm{Glimm}(A) \times \mathrm{Glimm} (B) , \tau_{cr} \right) \right),
\]
sending $f \mapsto f \circ \phi^{\beta}$.  

The situation is summarised in the following diagram:

\[
\xymatrix{ C \left( \beta \mathrm{Glimm} (A) \times \beta \mathrm{Glimm} (B) \right) \ar@{^{(}->}^{(\phi^{\beta})^{\ast}}[r] \ar_{\nu}[d] & C \left( \beta( \mathrm{Glimm}(A) \times \mathrm{Glimm} (B), \tau_{cr}) \right) \ar^{\psi^{\ast}}[d]  \\
C \left( \beta \mathrm{Glimm} (A) \right) \otimes C \left( \beta \mathrm{Glimm} (B) \right) \ar_{\mu_A \otimes \mu_B}[d] & C \left( \beta \mathrm{Glimm} ( A \otimes_{\alpha} B ) \right) \ar^{\mu_{\alpha}}[d]\\
ZM(A) \otimes ZM(B) \ar@{^{(}->}[r] & ZM(A \otimes_{\alpha} B)
}
\]

\begin{corollary}
\label{c:phibeta}
For any C$^{\ast}$-algebras $A$ and $B$, $ZM(A) \otimes ZM(B) = ZM( A \otimes_{\alpha} B )$ if and only if the canonical map 
\[
\phi^{\beta} : \beta \left( \mathrm{Glimm}(A) \times \mathrm{Glimm} (B) , \tau_{cr} \right) \rightarrow  \beta \mathrm{Glimm} (A) \times \beta \mathrm{Glimm} (B)
\]
is injective.  Moreover, when this occurs we have $\tau_p = \tau_{cr}$ on $\mathrm{Glimm}(A) \times \mathrm{Glimm} (B)$.
\end{corollary}
\begin{proof}
We first show that the preceding diagram commutes; that is, for any $h \in C( \beta \mathrm{Glimm}(A) \times \beta \mathrm{Glimm}(B) )$ the multipliers $z_1$ and $z_2$ of $A \otimes_{\alpha} B $ given by
\[
z_1 = ( \mu_{\alpha} \circ \psi^{\ast} \circ ( \phi^{\beta} )^{\ast} ) (h), \ z_2 = \left( ( \mu_A \otimes \mu_B ) \circ \nu \right) (h)
\]
are equal.  By linearity and continuity it suffices to check equality for functions of the form $h = \nu^{-1} (f \otimes g)$, where $f \in C ( \beta \mathrm{Glimm}(A) )$, $g \in C ( \beta \mathrm{Glimm}(B) )$.

Consider an elementary tensor $a \otimes b \in A \otimes_{\alpha} B $, and a pair $(P,Q) \in \mathrm{Prim} (A) \times \mathrm{Prim}(B)$. We will show that $(z_1 - z_2 ) (a \otimes b ) + \Phi (P,Q) = 0$.  Set $(p,q) = ( \rho_A \times \rho_B )(P,Q)$, and note that by Theorem~\ref{t:homeo} $(\psi \circ \rho_{\alpha} \circ \Phi )(P,Q) = (p,q)$. In particular it follows that
\[
( \psi^{\ast} \circ (\phi^{\beta} )^{\ast} )(h) (\rho_{\alpha} ( \Phi (P,Q) ) ) = (\phi^{\beta})^{\ast} (h) ( \psi \circ \rho_{\alpha} \circ \Phi (P,Q) ) = h (p,q) = f(p)g(q).
\]

Firstly by applying the $\ast$-isomorphism of Corollary~\ref{c:dh} to the element $( \psi^{\ast} \circ (\phi^{\beta})^{\ast} )(h)$ of $C( \beta ( \mathrm{Glimm}(A \otimes_{\alpha} B))$ see that
\begin{eqnarray*}
z_1 (a \otimes b) + \Phi (P,Q) & = & ( \mu_{\alpha} \circ \psi^{\ast} \circ ( \phi^{\beta} )^{\ast} ) (h) (a \otimes b ) + \Phi (P,Q) \\
 & = & f(p) g(q) \left( a \otimes b  \right) + \Phi (P,Q)\\
 & = & f(p) a \otimes g(q) b + \Phi (P,Q).
\end{eqnarray*}

On the other hand, applying Corollary~\ref{c:dh} to $f \in C( \beta \mathrm{Glimm}(A) )$ gives $\mu_A (f) a - f(p)a \in P$, so that
\[
(\mu_A (f) a - f(p) a ) \otimes g(q) b = \mu_A (f) a \otimes g(q) b - f(p)a \otimes g(q) b \in \mathrm{ker} (q_P \otimes q_Q ) = \Phi (P,Q),
\] 
from which it follows that $ \mu_A (f) a \otimes g(q) b + \Phi (P,Q) = f(p)a \otimes g(q)b + \Phi (P,Q)$.  A similar argument applied to $B$ gives
 $\mu_A (f) a \otimes g(q)b + \Phi (P,Q) = \mu_A (f) a \otimes \mu_B (g) b + \Phi (P,Q)$, and we conclude that
\begin{eqnarray*}
z_1 ( a \otimes b ) + \Phi (P,Q) & = & \mu_A (f) a \otimes \mu_B (g) b + \Phi (P,Q) \\
& = & (\mu_A \otimes \mu_B )(f \otimes g ) (a \otimes b ) + \Phi (P,Q) \\
& = & (\mu_A \otimes \mu_B \circ \nu )(h )(a \otimes b ) + \Phi (P,Q) \\
& = & z_2 (a \otimes b ) + \Phi (P,Q).
\end{eqnarray*}

 In particular $(z_1 - z_2)( a \otimes b ) \in \Phi (P,Q)$ for all $a \in A, b \in B$ and $(P,Q) \in \mathrm{Prim}(A) \times \mathrm{Prim}(B)$.  Since $\bigcap \{ \Phi (P,Q) : (P, Q ) \in \mathrm{Prim}(A) \times \mathrm{Prim}(B) \} = \{ 0 \}$, it follows that $(z_1 - z_2 ) (a \otimes b ) = 0$ for all $a \in A, b \in B$.  Thus $(z_1 - z_2 ) \left( A \otimes_{\alpha} B \right) = \{ 0 \}$, that is, $z_1 = z_2$.

Since the vertical arrows of the diagram all describe $\ast$-isomorphisms, the inclusion $ZM(A) \otimes ZM(B) \subseteq ZM(A \otimes_{\alpha} B)$ will be surjective if and only if $(\phi^{\beta})^{\ast}$ is.  By~\cite[Theorem 10.3]{gill_jer},  $(\phi^{\beta})^{\ast}$ is surjective if and only if $\phi^{\beta}$ is a homeomorphism.  

But then $\phi^{\beta}$, being a continuous surjection from a compact Hausdorff space to a Hausdorff space, is thus a homeomorphism if and only if it is injective.

The final assertion follows from the fact that $\phi^{\beta}$ is the identity on $\mathrm{Glimm}(A) \times \mathrm{Glimm} (B)$.

\end{proof}

Let $X$ and $Y$ be completely regular spaces.  The question of establishing conditions on $X,Y$ and $X \times Y$ for which canonical surjection $\iota^{\beta} : \beta(X \times Y) \rightarrow \beta X \times \beta Y$ is injective (and hence a homeomorphism) has been studied by several authors.  If either $X$ or $Y$ is finite, then this is trivially true. The most well-known characterisation  in the infinite case is due to Glicksberg~\cite{glicksberg}.

\begin{definition}
Let $X$ be a completely regular space.  We say that $X$ is \emph{pseudocompact} if every $f \in C(X)$ is bounded.
\end{definition}


Glicksberg's Theorem~\cite[Theorem 1]{glicksberg} states that, for infinite completely regular spaces $X$ and $Y$ the canonical map $\beta(X \times Y) \rightarrow \beta X \times \beta Y$ is a homeomorphism if and only if $X \times Y$ is pseudocompact.

\begin{theorem}
\label{p:zma}
For any C$^{\ast}$-algebras $A$ and $B$, $ZM(A) \otimes ZM(B) = ZM(A \otimes_{\alpha} B)$ if and only if one of the following conditions hold:
\begin{enumerate}
\item[(i)] $\mathrm{Glimm}(A)$ or $\mathrm{Glimm} (B)$ is finite, or
\item[(ii)] $\tau_p = \tau_{cr}$ on $\mathrm{Glimm}(A) \times \mathrm{Glimm} (B)$ and $\mathrm{Glimm}(A) \times \mathrm{Glimm} (B)$ is pseudocompact.
\end{enumerate}
\end{theorem}
\begin{proof}
If (i) holds, w.l.o.g. $\mathrm{Glimm}(B)$ is finite, hence discrete and compact.  In particular $\rho_B$ is an open map, so by Proposition~\ref{p:wcpct}(ii), $\tau_p = \tau_{cr}$ on $\mathrm{Glimm} (A) \times \mathrm{Glimm} (B)$.  Then by~\cite[Proposition 8.2]{walker} the map $\phi^{\beta}$ is a homeomorphism and hence $ZM(A) \otimes ZM(B) = ZM( A \otimes_{\alpha} B )$ by Corollary~\ref{c:phibeta}.

If $\mathrm{Glimm} (A)$ and $\mathrm{Glimm} (B)$ are infinite then by~\cite[Theorem 1]{glicksberg}, $\beta  \left( ( \mathrm{Glimm}(A) \times \mathrm{Glimm} (B) , \tau_p ) \right)$ is canonically homeomorphic to $\beta \mathrm{Glimm}(A) \times \beta \mathrm{Glimm} (B)$ if and only if $\left( \mathrm{Glimm}(A) \times \mathrm{Glimm} (B) , \tau_p \right)$ is pseudocompact.  Hence in the infinite case, Corollary~\ref{c:phibeta} gives $ZM(A) \otimes ZM(B) = ZM(A \otimes_{\alpha} B)$ if and only if (ii) holds.
\end{proof}

Clearly if $M(A) \otimes_{\alpha} M(B) = M(A \otimes_{\alpha} B )$ then $ZM(A) \otimes ZM(B) = Z( M(A) \otimes_{\alpha} M(B) ) = ZM(A \otimes_{\alpha} B )$.  We will show in Example~\ref{e:zma} that the converse is not true; we construct C$^{\ast}$-algebras $A$ and $B$ such that $ZM(A) \otimes ZM(B) = Z(M(A \otimes_{\alpha} B ) )$, but $M(A) \otimes_{\alpha} M(B) \subsetneq M(A \otimes_{\alpha} B )$.
\begin{remark}
It is easily seen that the continuous image of a pseudocompact space is pseudocompact. In particular if $X$ and $Y$ are completely regular spaces such that $X \times Y$ is pseudocompact, then since the projection maps $\pi_X$ and $\pi_Y$ are continuous we have necessarily that both $X$ and $Y$ are pseudocompact.

In the other direction, it is not always true that a product of pseudocompact spaces is pseudocompact; see~\cite[Example 9.15]{gill_jer} for a counterexample.  However, for a product of pseudocompact spaces $X$ and $Y$, one of which is also locally compact, then $X \times Y$ is pseudocompact  by~\cite[Proposition 8.21]{walker}.

\end{remark}
\begin{remark}

In the particular case that $A$ and $B$ are unital, then $\mathrm{Prim}(A)$ and $\mathrm{Prim}(B)$ are compact so that $\tau_{cr}= \tau_p$ on $\mathrm{Glimm}(A) \times \mathrm{Glimm} (B)$ by Proposition~\ref{p:wcpct}(i). Moreover, $\mathrm{Glimm}(A) \times \mathrm{Glimm}(B)$ is compact, hence pseudocompact.  Thus Theorem~\ref{p:zma} implies the Haydon-Wassermann result~\cite[Corollary 1]{hay_wass} in the unital case.
\end{remark}

\section{Glimm  ideals}

We now turn to the question of determining the  Glimm ideals of $A \otimes_{\alpha} B$ in terms of those of$A$ and $B$.  More precisely Theorem~\ref{t:homeo2} shows that, when the Glimm spaces considered as sets of ideals of $A$,$B$ and $A \otimes_{\alpha} B$, then the map $\Delta$ of Equation~(\ref{e:delta}) satisfies $\Delta = \psi^{-1}$. 

We define a new map $\Psi:\mathrm{Id}' (A \otimes_{\alpha} B) \rightarrow \mathrm{Id}'(A) \times \mathrm{Id}'(B)$, which is  a left inverse of the map $\Phi$ of Equation~(\ref{e:phi}).  
For $M \in \mathrm{Id} '( A \otimes_{\alpha} B ) $ we define closed two-sided ideals $M^A$ and $M^B$ of $A$ and $B$ respectively via
\begin{equation}
\label{e:psi}
M^A = \{ a \in A : a \otimes B \subseteq M \}, M^B = \{ b \in B : A \otimes b \subseteq M \}.
\end{equation}
The assignment $\Psi (M) = (M^A , M^B )$ gives a map $\Psi : \mathrm{Id}' ( A \otimes_{\alpha} B ) \rightarrow \mathrm{Id}' (A) \times \mathrm{Id}'(B)  $.    

\begin{proposition}
\label{p:psi}
Let $A$ and $B$ be C$^{\ast}$-algebras and $A \otimes_{\alpha} B$ their minimal C$^{\ast}$-tensor product.  Then the map $\Psi: \mathrm{Id}'(A \otimes_{\alpha} B ) \rightarrow \mathrm{Id}'(A) \times \mathrm{Id}' (B)$ satisfies the following properties:
\begin{enumerate}
\item[(i)] $\Psi \circ \Phi$ is the identity on $\mathrm{Id}'(A) \times \mathrm{Id}' (B)$,
\item[(ii)] $\Psi ( \mathrm{Fac} (A \otimes_{\alpha} B ) ) = \mathrm{Fac}(A) \times \mathrm{Fac} (B)$, 
\item[(iii)] The restriction of $\Psi$ to $\mathrm{Fac}(A \otimes_{\alpha} B )$ is continuous in the hull-kernel topologies,
\item[(iv)] For any $M \in \mathrm{Fac}(A \otimes_{\alpha} B)$, the inclusion $M \subseteq \Phi \circ \Psi (M)$ holds.
\end{enumerate}
\end{proposition}
\begin{proof}
(i) and (iii) are shown in the proof of~\cite[Theorem 2.6]{lazar_tensor}.  To prove (iii),~\cite[Proposition 1]{guichardet} shows that $\Psi( \mathrm{Fac}(A \otimes_{\alpha} B) \subseteq \mathrm{Fac} (A) \times \mathrm{Fac} (B)$.  Surjectivity then follows from Proposition~\ref{p:phi}(iii) and part (ii).

As for (iv), it is shown in~\cite[Lemma 2.13(iv)]{blanch_kirch} that for any prime ideal $M$ of $A \otimes_{\alpha} B$ we have $M \subseteq \Phi \circ \Psi (M)$. But then~\cite[Proposition II.6.1.11]{black} shows that every factorial ideal of a C$^{\ast}$-algebra is prime, from which (iv) follows.

\end{proof}

We remark that Proposition~\ref{p:psi} (ii) shows that $\Psi$ maps $\mathrm{Prim} ( A \otimes_{\alpha} B )$ to $\mathrm{Fac} (A) \times \mathrm{Fac} (B) $. It is not known in general whether $\Psi$ maps $\mathrm{Prim} (A \otimes_{\alpha} B )$ onto $\mathrm{Prim} (A) \times \mathrm{Prim} (B)$.  For this reason, we will need to use an alternative construction of the space of Glimm ideals of a C$^{\ast}$-algebra, which was first considered by Kaniuth in~\cite{kaniuth}.  

It is shown in~\cite[Section 2]{kaniuth} how for any C$^{\ast}$-algebra $A$, $\mathrm{Glimm} (A)$ can be constructed as $\rho \left( \mathrm{Fac} (A) \right)$.  For $I,J \in \mathrm{Fac} (A)$ we write $I \approx_f J$ if $f(I)=f(J)$ for all $f \in C^b ( \mathrm{Fac} (A) )$, and denote by $[I]_f$ the equivalence class of $I$ in $\mathrm{Fac} (A)$.  

\begin{proposition}
\label{p:approxf}
Let $A$ be a C$^{\ast}$-algebra.  Then the relation $\approx_f$ on $\mathrm{Fac}(A)$ has the following properties:
\begin{enumerate}
\item[(i)] For $I \in \mathrm{Fac}(A)$ and $P \in \mathrm{hull} (I)$ we have
\[
[I]_f \cap \mathrm{Prim}(A) = [P] \mbox{ and } k \left( [I]_f \right) = k \left( [P] \right),
\]
\item[(ii)] $\mathrm{Fac}(A) / \approx_f $ is homeomorphic to $\mathrm{Prim}(A) / \approx$ via the map $[I]_f \mapsto [P]$, where $P \in \mathrm{hull} (I)$, when both spaces are considered with the quotient topology,
\item[(iii)] Each Glimm ideal of $A$ is of the form $G_I = k \left( [I]_f \right)$ for some $I \in \mathrm{Fac}(A)$.
\item[(iv)] The equivalence classes of $\approx_f$ satisfy
\[
[I]_f = \mathrm{hull}_f (G_I ).
\]
\end{enumerate}
\end{proposition}
\begin{proof}
Parts (i) and (ii) are shown in~\cite[Lemma 2.2]{kaniuth}.  (iii) is immediate from (i).

To prove (iv) take $I \in \mathrm{Fac}(A)$.  It follows from the definition of $\approx_f$ that the equivalence class $[I]_f$ is a closed subset of $\mathrm{Fac}(A)$.  By the definition of the hull-kernel topology and by part (iii) we then have
\[
[I]_f = \mathrm{hull}_f \left( k([I]_f) \right) = \mathrm{hull}_f (G_I).
\]
\end{proof}

As a consequence of Proposition~\ref{p:approxf}(ii), we shall consider the set of equivalence classes $\mathrm{Fac}(A) / \approx_f$ as $\mathrm{Glimm}(A)$, and denote by $\rho_{A}^f : \mathrm{Fac}(A) \rightarrow \mathrm{Glimm}(A)$ the corresponding quotient map.  Moreover, we may unambiguously speak of the quotient topology $\tau_q$ on $\mathrm{Glimm}(A)$ as the strongest topology on this space for for which either $\rho_A$ or $\rho_A^f$ is continuous.

For C$^{\ast}$-algebras $A$ and $B$ and two pairs of ideals $(P,Q),(R,S) \in \mathrm{Fac} (A) \times \mathrm{Fac} (B)$, we will write $(P,Q) \approx_f (R,S)$ when $g (P,Q) = g (R,S)$ for all $g \in C^b ( \mathrm{Fac} (A) \times \mathrm{Fac} (B) )$. By Lemma~\ref{l:prodcr}, this is equivalent to saying $P \approx_f R$ and $Q \approx_f S$.

Lemmas~\ref{l:phiglimm},~\ref{l:phicircpsi} and Proposition~\ref{p:hullf} below relate equivalence classes of the relation $\approx_f$ in $\mathrm{Fac}(A) \times \mathrm{Fac} (B)$ with those in $\mathrm{Fac} (A \otimes_{\alpha} B )$, via the maps $\Phi$ and $\Psi$.
\begin{lemma}
\label{l:phiglimm}
Let $(I,J) \in \mathrm{Fac} (A) \times \mathrm{Fac} (B)$, and denote by $G_I =k \left( [I]_f \right)$ and $G_J = k \left( [J]_f \right)$ the corresponding Glimm ideals of $A$ and $B$ respectively.  Then $G_{\Phi(I,J)}:= k \left( [ \Phi(I,J) ]_f \right)$, the Glimm ideal of $A \otimes_{\alpha} B$ corresponding to $[\Phi (I,J)]_f$, satisfies $ G_{\Phi (I,J)} \subseteq \Phi (G_I,G_J) $.
\end{lemma}
\begin{proof}
The fact that $\Phi (I,J) \in \mathrm{Fac} (A \otimes_{\alpha} B )$ follows from Proposition~\ref{p:phi} (iii). Taking $(P,Q) \in \mathrm{hull} (I) \times \mathrm{hull} (J)$ we have $\Phi (P,Q) \in \mathrm{hull} \left( \Phi (I,J) \right)$ by Proposition~\ref{p:phi} (i). In this case Proposition~\ref{p:approxf}(i) gives
\[
k([I]_f) = k ( [P] ), k([J]_f)=k([Q])\mbox{ and } k \left( [ \Phi (I,J) ]_f \right) =  k \left( [ \Phi (P,Q) ] \right), 
\]  
and so we may replace $(I,J)$ with $(P,Q)$.

Note that if $(R,S) \in \mathrm{Prim}(A) \times \mathrm{Prim} (B)$ such that $(R,S) \approx (P,Q)$, and $f \in C^b ( \mathrm{Prim}( A \otimes_{\alpha} B)$ then $f \circ \Phi \in C^b ( \mathrm{Prim} (A) \times \mathrm{Prim} (B))$, hence $f(\Phi(R,S))=f( \Phi (P,Q))$, so that $\Phi (R,S) \approx \Phi (P,Q)$.  It then follows from Proposition~\ref{p:phi}(v) that
\begin{eqnarray*}
\Phi \left( G_I,G_J \right) = \Phi \left( k([P]),k([Q]) \right) & = & \bigcap \{ \Phi (R,S): (R,S) \in \mathrm{hull}(k[P]) \times \mathrm{hull}(k[Q]) \} \\
& = & \bigcap \{ \Phi(R,S) :  (R,S) \approx  (P,Q) \} \\
& \supseteq & \bigcap \{ M \in \mathrm{Prim} ( A \otimes_{\alpha} B ) : M \approx \Phi(P,Q) \} \\
&=& k \left( [ \Phi(P,Q) ] \right) =G_{\Phi(I,J)} \\
\end{eqnarray*}
\end{proof}
\begin{lemma}
\label{l:phicircpsi}
For any $M \in \mathrm{Fac}(A \otimes_{\alpha} B)$, $\Phi \circ \Psi (M) \in \mathrm{Fac}(A \otimes_{\alpha} B)$ and $M \approx_f \Phi \circ \Psi (M)$.
\end{lemma}
\begin{proof}
The fact that $\Phi \circ \Psi (M) \in \mathrm{Fac}(A \otimes_{\alpha} B)$ follows from Propositions~\ref{p:phi}(iii) and~\ref{p:psi}(ii).  By Proposition~\ref{p:psi}(iv), we have $M \subseteq \Phi \circ \Psi (M)$.  Hence $\Phi \circ \Psi (M) \in \mathrm{hull}_f (M) = \overline{ \{ M \} }$, so that $M \approx_f \Phi \circ \Psi (M)$.

\end{proof}

Note that the proof of Proposition~\ref{p:hullf} below requires that we base the definition of Glimm ideals on the complete regularisation of the space of factorial ideals (since $\Psi$ maps factorial ideals to factorial ideals).

\begin{proposition}
\label{p:hullf}
Let $(I,J) \in \mathrm{Fac} (A) \times \mathrm{Fac} (B) $, $M \in \mathrm{Fac} ( A \otimes_{\alpha} B )$ and denote by $(M^A,M^B)=\Psi (M)$.  Then $M \approx_f \Phi (I,J)$ if and only if $ (M^A,M^B) \approx_f (I,J)$. Hence with $G_I,G_J$ and $G_{\phi (I,J)}$ as defined in Lemma~\ref{l:phiglimm}, we have
\[
M \in \mathrm{hull}_f \left( G_{\Phi(I,J)} \right) \mbox{ if and only if } (M^A,M^B) \in \mathrm{hull}_f (G_I) \times \mathrm{hull}_f (G_J)
\]
\end{proposition}
\begin{proof}
Suppose $M \approx_f \Phi (I,J)$, and take $g \in C^b ( \mathrm{Fac} (A) \times \mathrm{Fac} (B) )$. Using Proposition~\ref{p:psi} (ii) and (iii), we have $g \circ \Psi \in C^b \left( \mathrm{Fac} (A \otimes_{\alpha} B )\right)$. Hence
\[
g(M^A,M^B) = (g \circ \Psi)(M) = (g \circ \Psi ) ( \Phi (I,J) ) = g(I,J), 
\] 
since $\Psi \circ \Phi$ if the identity on $\mathrm{Fac} (A) \times \mathrm{Fac} (B)$ by Proposition~\ref{p:psi} (i).  It follows that $(M^A,M^B) \approx_f (I,J)$.

Since  $G_{\Phi(I,J)}=k \left( [ \Phi(I,J) ]_f \right)$, Proposition~\ref{p:approxf}(iv) shows that $[ \Phi (I,J) ]_f = \mathrm{hull}_f (G_{\Phi (I,J)})$.  Similarly $[I]_f = \mathrm{hull}_f (G_I)$ and $[J]_f = \mathrm{hull}_f (G_J)$.

To prove the converse, suppose that $(M^A, M^B) \approx_f (I,J)$.  Then by Lemma~\ref{l:prodcr}, $M^A \approx_f I$ and $M^B \approx_f J$, so that $M^A \supseteq G_I$ and $M^B \supseteq G_J$. Together with Lemma~\ref{l:phiglimm}, this gives the inclusion
\[
\Phi \circ \Psi (M) = \Phi (M^A,M^B) \supseteq \Phi \left( G_I,G_J \right) \supseteq G_{\Phi (I,J)},
\]
and since by Proposition~\ref{p:phi}(iii) $\Phi \circ \Psi (M) \in \mathrm{Fac} (A \otimes_{\alpha} B)$, it follows from Proposition~\ref{p:approxf}(iv) that $\Phi \circ \Psi (M) \approx_f \Phi (I,J)$.  Then by Lemma~\ref{l:phicircpsi} we have $M \approx_f \Phi (I,J)$.

The final assertion of the statement follows from Proposition~\ref{p:approxf}(iv).
\end{proof}

In what follows we make use of the map $\Delta: \mathrm{Id}' (A) \times \mathrm{Id}' (B) \rightarrow \mathrm{Id}' (A \otimes_{\alpha} B )$ defined via~(\ref{e:delta}).  For  $(I,J) \in \mathrm{prime} (A) \times \mathrm{prime} (B)$, $(\Psi \circ \Delta) (I,J)=(I,J)$~\cite[Lemma 2.13(i)]{blanch_kirch}. We will extend this to general $(I,J) \in \mathrm{Id}'(A) \times \mathrm{Id}'(B)$ in Lemma~\ref{l:hulld} below.   On the other hand, if $K \in \mathrm{Id}' ( A \otimes_{\alpha} B )$ then 
\begin{equation}
\label{e:deltacircpsi}
(\Delta \circ \Psi) (K) = K^A \otimes_{\alpha} B + A \otimes_{\alpha} K^B \subseteq K
\end{equation}
by the definition of $K^A$ and $K^B$ in~(\ref{e:psi}).
\begin{lemma}
\label{l:hulld}
Let $I$ and $J$ be proper ideals of $A$ and $B$ respectively.  Then 
\begin{enumerate}
\item[(i)] $ \mathrm{hull}_f \Delta (I,J) =  \Psi^{-1} ( \mathrm{hull}_f (I) \times \mathrm{hull}_f (J) )$.
\item[(ii)] $\Psi \circ \Delta (I,J) = (I,J) $
\end{enumerate}
\end{lemma}
\begin{proof}
To show (i), take $F \in \mathrm{hull}_f \Delta (I,J)$, then $\Psi (F) = (F^A,F^B ) \in \mathrm{Fac} (A) \times \mathrm{Fac} (B)$ and $F^A \supseteq I, F^B \supseteq J$.  Hence $\Psi (F) \in \mathrm{hull}_f (I) \times \mathrm{hull}_f (J)$.

On the other hand, suppose $F \in \mathrm{Fac} (A \otimes_{\alpha} B )$ and $ \Psi (F) \in \mathrm{hull}_f (I) \times \mathrm{hull}_f (J)$.  Then using~(\ref{e:deltacircpsi}), $\Delta (I,J)\subseteq \Delta (F^A,F^B) \subseteq F$, as required.

To prove (ii), denote by $K = \Delta (I,J)$ and $(K^A , K^B ) = \Psi (K)$.   Then if $a \in I$, $a \otimes B \subseteq I \otimes_{\alpha} B \subseteq K$, so $a \in K^A$ and hence $I \subseteq K^A$.  On the other hand, suppose $a \in K^A$ and $b \in B \backslash J$, so that $a \otimes b \in K$.  Choose a bounded linear functional $\lambda$ on $B$ vanishing on $J$ such that $\lambda (b) = 1$.  Let $L_{\lambda} : A \otimes_{\alpha} B \rightarrow A$ be the corresponding left slice map defined via $L_{\lambda} (a' \otimes b' ) = \lambda (b') a'$ on elementary tensors and extended to $A \otimes_{\alpha} B$ by linearity and continuity~\cite[Theorem 1]{tomiyama}.  Then $L_{\lambda} ( A \odot J ) = \{ 0 \}$ and $L_{\lambda} (I \odot B ) \subseteq I$, so that $L_{\lambda} (K) \subseteq I$. In particular $L_{\lambda} (a \otimes b ) =a \in I$, hence $K^A \subseteq I$ and so $K^A=I$. A similar argument shows that $K^B = J$, which completes the proof.
\end{proof}

\begin{corollary}
\label{c:delta}
Let $(I,J) \in \mathrm{Fac} (A) \times \mathrm{Fac} (B)$.  Then with $G_I,G_J$ and $G_{\Phi (I,J)}$ defined as in Lemma~\ref{l:phiglimm}, we have 
\[
 G_{\Phi (I,J)} = G_I \otimes_{\alpha} B + A \otimes_{\alpha} G_J
\]
\end{corollary}
\begin{proof}
Take $M \in \mathrm{Fac} ( A \otimes_{\alpha} B )$.  By Proposition~\ref{p:hullf}, $M \supseteq G_{\Phi (I,J)}$ if and only if \\ $M \in \Psi^{-1} \left( \mathrm{hull}_f (G_I) \times \mathrm{hull}_f (G_J )\right)$.  Hence by Lemma~\ref{l:hulld}(i), $M \supseteq G_{\Phi (I,J)}$ if and only if $M \supseteq \Delta \left( G_I,G_J \right) $, so $G_{\Phi (I,J)} = \Delta \left( G_I, G_J \right)$.
\end{proof}

As a consequence of Corollary~\ref{c:delta}, we are now in a position to prove a similar result to~\cite[Theorem 2.3]{kaniuth}:
\begin{theorem}
\label{t:homeo2}

Let $A$ and $B$ be C$^{\ast}$-algebras and denote by $\psi : \mathrm{Glimm}(A \otimes_{\alpha} B) \rightarrow \left( \mathrm{Glimm}(A) \times \mathrm{Glimm}(B) , \tau_{cr} \right)$ the homeomorphism of Theorem~\ref{t:homeo}.  Then identifying the Glimm spaces with the corresponding sets of ideals we have
\begin{enumerate}
\item[(i)] $\Delta = \psi^{-1}$ on $\mathrm{Glimm}(A) \times \mathrm{Glimm}(B)$, hence $\Delta$ is a homeomorphism of $\left( \mathrm{Glimm}(A) \times \mathrm{Glimm} (B) , \tau_{cr} \right)$ onto $\mathrm{Glimm}(A \otimes_{\alpha} B )$,
\item[(ii)] $\psi$ is given by the restriction of $\Psi$ to $\mathrm{Glimm}(A \otimes_{\alpha} B )$.

\end{enumerate}
\end{theorem}
\begin{proof}
Following the notation of Theorem~\ref{t:homeo}, Proposition~\ref{p:approxf}(i) and Corollary~\ref{c:delta} show that the diagram
\[
\xymatrix{ \mathrm{Prim} (A) \times \mathrm{Prim} (B) \ar@{^{(}->}^{\Phi}[r] \ar_{\rho_A \times \rho_B }[d] & \mathrm{Prim} (A \otimes_{\alpha} B ) \ar^{\rho_{\alpha}}[d] \\
\mathrm{Glimm} (A) \times \mathrm{Glimm} (B) \ar^{\Delta}[r] &  \mathrm{Glimm} ( A \otimes_{\alpha} B ) 
}
\]
commutes, i.e., that $\Delta \circ ( \rho_A \times \rho_B) = \rho_{\alpha} \circ \Phi$.  Therefore if we can show that $\psi^{-1} \circ ( \rho_A \times \rho_B ) = \rho_{\alpha} \circ \Phi$ also, then it will follow necessarily that $\Delta = \psi^{-1}$ (since $\rho_A \times \rho_B$ is surjective).

From Theorem~\ref{t:homeo}, $\overline{(\rho_A \times \rho_B)} \circ \Phi = \rho_A \times \rho_B$ and $\psi \circ \rho_{\alpha} = \overline{(\rho_A \times \rho_B )}$, so we have 
\begin{eqnarray*}
\psi^{-1} \circ ( \rho_A \times \rho_B ) &=& \psi^{-1} \circ \overline{(\rho_A \times \rho_B)} \circ \Phi \\
&=& \psi^{-1} \circ ( \psi \circ \rho_{\alpha} ) \circ \Phi \\
& = & \rho_{\alpha} \circ \Phi, \\
\end{eqnarray*}
which proves (i).

Assertion (ii) is then immediate from (i) and Lemma~\ref{l:hulld}(ii).
\end{proof}

Suppose that $A$ and $B$ are C$^{\ast}$-algebras such that $A \otimes_{\alpha} B$ satisfies Tomiyama's property (F). Under this assumption, Kaniuth's result~\cite[Theorem 2.3]{kaniuth} shows that the map $\Delta : \left( \mathrm{Glimm}(A), \tau_q \right) \times \left( \mathrm{Glimm} (B),\tau_q \right) \rightarrow \left( \mathrm{Glimm} (A \otimes_{\alpha} B ) , \tau_q \right)$ is an open bijection, where $\tau_q$ denotes the quotient topology on the Glimm space as discussed in Remark~\ref{r:tauq}.  

In  order to extend this to arbitrary minimal tensor  products, we first need some Lemmas:

\begin{lemma}
\label{l:tauq}
Suppose that the complete regularisation maps $\rho_A$ and $\rho_B$ are open with respect to either $\tau_q$ or $\tau_{cr}$ on $\mathrm{Glimm}(A)$ and $\mathrm{Glimm}(B)$ respectively.  Then
\begin{enumerate}
\item[(i)] $\tau_q = \tau_{cr}$ on each of $\mathrm{Glimm}(A)$ and $\mathrm{Glimm}(B)$,
\item[(ii)] $\tilde{\tau}_q = \tau_{cr} = \tau_p$ on $\mathrm{Glimm}(A) \times \mathrm{Glimm} (B)$, where $\tilde{\tau}_q$ is the topology induced by the product map $\rho_A \times \rho_B$,
\item[(iii)] $\tau_q = \tau_{cr}$ on $\mathrm{Glimm}(A \otimes_{\alpha} B )$.
\end{enumerate}
\end{lemma}
\begin{proof}
(i) is shown in the discussion in~\cite[p. 351]{arch_som_qs}.

(ii) As a consequence of (i), we may consider $\tau_p$ as the product of the quotient topologies.  Since $\rho_A \times \rho_B$ is necessarily $\tau_p$ continuous, $\tau_p$ is always weaker than $\tilde{\tau_q}$.  

Consider a $\tilde{\tau}_q$ open subset $\mathcal{U}$ of $\mathrm{Glimm}(A) \times \mathrm{Glimm}(B)$ and let $(x,y) \in \mathcal{U}$.  Choose $(P,Q) \in \mathrm{Prim}(A) \times \mathrm{Prim} (B)$ with $(\rho_A \times \rho_B )(P,Q) = (x,y)$.  Then since $(\rho_A \times \rho_B )^{-1} ( \mathcal{U})$ is an open subset of $\mathrm{Prim}(A) \times \mathrm{Prim} (B)$, we can find open neighbourhoods $\mathcal{W}$ of $P$ and $\mathcal{S}$ of $Q$ such that $\mathcal{W} \times \mathcal{S} \subseteq (\rho_A \times \rho_B )^{-1} ( \mathcal{U}  )$.  Then we have
\[
(\rho_A \times \rho_B )(P,Q) = (x,y) \in \rho_A ( \mathcal{W} ) \times \rho_B ( \mathcal{S} ) \subseteq \mathcal{U} ,
\]
and $\rho_A ( \mathcal{W} ) \times \rho_B ( \mathcal{S} )$ is $\tau_p$-open  since $\rho_A$ and $\rho_B$ are both $\tau_q$-open.  In particular $(x,y)$ is a $\tau_p$-interior point of $\mathcal{U}$, and  hence $\mathcal{U}$ is $\tau_p$-open.

The fact that $\tau_p = \tau_{cr}$ follows from condition (ii) of Proposition~\ref{p:wcpct}.

As for (iii), it is always true that $\tau_q$ is stronger than $\tau_{cr}$, thus we need to prove that any $\tau_q$ open subset $\mathcal{U}$ of $\mathrm{Glimm}(A \otimes_{\alpha} B )$ is $\tau_{cr}$-open.     By part (ii) and Corollary~\ref{c:wcpct2}, $\psi$ is a homeomorphism of  $\left( \mathrm{Glimm}(A \otimes_{\alpha} B) , \tau_{cr} \right)$ onto $(\mathrm{Glimm}(A) \times  \mathrm{Glimm}(B) , \tilde{\tau}_q )$.  So given a $\tau_q$-open subset $\mathcal{U}$ of $\mathrm{Glimm}(A \otimes_{\alpha} B )$, it will suffice to prove that $\psi (\mathcal{U})$ is $\tilde{\tau}_q$-open, that is, that $(\rho_A \times \rho_B)^{-1} ( \psi ( \mathcal{U} ) )$ is open.

Denote by $\mathcal{W}= (\rho_A \times \rho_B )^{-1} ( \psi ( \mathcal{U} ) )$.  We will show that $\mathcal{W}  = \Phi^{-1} \left( \rho_{\alpha}^{-1} ( \mathcal{U} ) \right)$. Since $\mathcal{U}$ is $\tau_q$-open and $\Phi$ is continuous on $\mathrm{Prim}(A) \times \mathrm{Prim}(B)$ by Proposition~\ref{p:phi}(ii), this will imply that $\mathcal{W}$ is open. For any $(P , Q ) \in \mathrm{Prim}(A) \times \mathrm{Prim} (B)$,  Theorem~\ref{t:homeo} gives
\[
\psi \circ \rho_{\alpha} \circ \Phi (P, Q ) = ( \rho_A \times \rho_B ) (P,Q),
\]
so that $(P,Q) \in \mathcal{W} $ if and only if $\rho_{\alpha} \circ \Phi (P,Q) \in \psi^{-1} ( \psi ( \mathcal{U} ))= \mathcal{U}$.  It follows that $\mathcal{W} = \Phi^{-1} \left( \rho_{\alpha}^{-1} ( \mathcal{U} ) \right)$, so that $\psi( \mathcal{U} )$ is $\tilde{\tau}_q$-open and hence $\mathcal{U}$ is $\tau_{cr}$-open.

\end{proof}

\begin{lemma}
\label{l:homeo}
Let $\rho_A^f, \rho_B^f$ and $\rho_{\alpha}^f$ denote the complete regularisation maps of $\mathrm{Fac}(A), \mathrm{Fac} (B)$ and $\mathrm{Fac}(A \otimes_{\alpha} B)$ respectively, and let $\psi: \mathrm{Glimm}(A \otimes_{\alpha} B) \rightarrow \mathrm{Glimm}(A) \times \mathrm{Glimm} (B)$ be the map of Theorem~\ref{t:homeo}.  Then it holds that
\[
\psi \circ \rho_{\alpha}^f \circ \Phi = \rho_{A}^f \times \rho_B^f
\]
on $\mathrm{Fac}(A) \times \mathrm{Fac} (B)$.
\end{lemma}
\begin{proof}
Let $(I,J) \in \mathrm{Fac}(A) \times \mathrm{Fac} (B)$ and denote by $(p,q) = ( \rho_A^f \times \rho_B^f )(I,J)$.  Take $(P,Q) \in \mathrm{hull}(I) \times \mathrm{hull} (J)$, so that $(\rho_A \times \rho_B )(P,Q) = (p,q)$ by Proposition~\ref{p:approxf}(i).  Then by Proposition~\ref{p:phi}(i), $\Phi (P,Q) \in \mathrm{hull} \left( \Phi (I,J) \right)$, so that  $( \rho_{\alpha} \circ \Phi )(P,Q) = ( \rho_{\alpha}^f \circ \Phi )(I,J)$ by Proposition~\ref{p:approxf}(i) applied to $A \otimes_{\alpha} B$.

Finally, Theorem~\ref{t:homeo} gives
\[
(\psi \circ \rho_{\alpha} \circ \Phi )(P,Q) = ( \rho_A \times \rho_B )(P,Q),
\]
so that
\[
( \psi \circ \rho_{\alpha}^f \circ \Phi )(I,J) = ( \rho_A^f \times \rho_B^f)(I,J),
\]
as required.
\end{proof}
We are now in a position to extend~\cite[Theorem 2.3]{kaniuth}, which required the assumption that $A \otimes_{\alpha} B$ satisfies property (F).
\begin{theorem}
\label{t:kaniuth}
The map $\psi$ of Theorem~\ref{t:homeo} defines a continuous bijection of $(\mathrm{Glimm}(A \otimes_{\alpha} B) , \tau_q )$ onto the product space $(\mathrm{Glimm}(A) , \tau_q) \times ( \mathrm{Glimm} (B) , \tau_q )$, where $\tau_q$ denotes the quotient topology induced by the complete regularisation map.  It follows that its inverse $\Delta$ is an open bijection.  Moreover, $\Delta$ is a homeomorphism whenever the complete regularisation maps $\rho_A$ and $\rho_B$ are open with respect to the quotient topologies on $\mathrm{Glimm}(A)$ and $\mathrm{Glimm}(B)$.
\end{theorem}
\begin{proof}
Let $\mathcal{U} \times \mathcal{V}$ be a basic open subset of the product space $(\mathrm{Glimm}(A), \tau_q) \times ( \mathrm{Glimm}(B) , \tau_q )$. Then by Proposition~\ref{p:approxf}(ii), the preimages $\mathcal{W} : = (\rho_A^f)^{-1} ( \mathcal{U} )$ and $\mathcal{S}:= (\rho_B^f)^{-1} ( \mathcal{V} )$ are open subsets of $\mathrm{Fac}(A)$ and $\mathrm{Fac}(B)$ respectively.  We claim that $\psi^{-1}( \mathcal{U} \times \mathcal{V} )$ is a $\tau_q$-open subset of $\mathrm{Glimm}(A \otimes_{\alpha} B )$, that is, that $(\rho_{\alpha}^f)^{-1} \left( \psi^{-1}(\mathcal{U} \times \mathcal{V} ) \right)$ is an open subset of $\mathrm{Fac}(A \otimes_{\alpha} B )$.  Since  the map $\Psi$ is continuous by Proposition~\ref{p:psi}(iii), it will suffice to show that
\[
(\rho_{\alpha}^f)^{-1} \left( \psi^{-1} ( \mathcal{U} \times \mathcal{V} ) \right) = \Psi^{-1} ( \mathcal{W} \times \mathcal{S} ).
\]

Let $M \in \Psi^{-1} ( \mathcal{W} \times \mathcal{S} )$.  By Lemma~\ref{l:phicircpsi} $M \approx_f \Phi \circ \Psi (M)$, so that $\rho_{\alpha}^f (M) = \rho_{\alpha}^f ( \Phi \circ \Psi (M))$. Then Lemma~\ref{l:homeo} gives
\[
\psi \circ \rho_{\alpha}^f (M) = ( \psi \circ \rho_{\alpha}^f \circ \Phi ) ( \Psi (M) ) = ( \rho_A^f \times \rho_B^f ) ( \Psi (M)) \in \mathcal{U} \times \mathcal{V}.
\]
Hence $\rho_{\alpha}^f (M) \in \psi^{-1} ( \mathcal{U} \times \mathcal{V} )$, and we have $\Psi^{-1} ( \mathcal{W} \times \mathcal{S} ) \subseteq ( \rho_{\alpha}^f)^{-1} \left( \psi^{-1} ( \mathcal{U} \times \mathcal{V} ) \right)$.

To show the reverse inclusion, let $M \in (\rho_{\alpha}^{f})^{-1} \left( \psi^{-1} ( \mathcal{U} \times \mathcal{V} ) \right)$, and denote by $(p,q) =  ( \psi \circ \rho_{\alpha}^f )(M) \in \mathcal{U} \times \mathcal{V}$.  Choose $(I,J) \in \mathcal{W} \times \mathcal{S}$ with $ (\rho_{A}^f \times \rho_{B}^f)(I,J) = (p,q)$.  Then invoking Lemma~\ref{l:homeo} again we have
\[
(\psi \circ \rho_{\alpha}^f)( \Phi (I,J) ) = ( \rho_A^f \times \rho_B^f)(I,J) = (p,q).
\]
Since $\psi$ is injective and $(\psi \circ \rho_{\alpha}^f)(M)=(p,q)$, it follows that $\rho_{\alpha}^f(M) = \rho_{\alpha}^f ( \Phi (I,J))$ and hence $M \approx_f \Phi (I,J)$.  By Proposition~\ref{p:hullf}, this implies that $\Psi (M) \approx_f (I,J)$, so that $( \rho_A^f \times \rho_B^f) ( \Psi (M) ) = (p,q)$.  In particular $\Psi(M) \in ( \rho_A^f \times \rho_B^f)^{-1}( \mathcal{U} \times \mathcal{V} ) = \mathcal{W} \times \mathcal{S}$, so that $M \in \Psi^{-1} ( \mathcal{W} \times \mathcal{S} )$ and hence $(\rho_{\alpha}^f)^{-1} \left( \psi^{-1} ( \mathcal{U} \times \mathcal{V} ) \right) \subseteq \Psi^{-1} ( \mathcal{W} \times \mathcal{S} )$, as required.

If in addition the complete regularisation maps $\rho_A$ and $\rho_B$ are open, then by Lemma~\ref{l:tauq} we have $\tau_q = \tau_{cr}$ on each of $\mathrm{Glimm}(A) , \mathrm{Glimm}(B)$ and $\mathrm{Glimm}(A \otimes_{\alpha} B )$.  Applying Corollary~\ref{c:wcpct2} and Theorem~\ref{t:homeo2} it follows that $\Delta$ is a homeomorphism of $( \mathrm{Glimm}(A) , \tau_q) \times ( \mathrm{Glimm}(B) , \tau_q )$ onto $\mathrm{Glimm}(A \otimes_{\alpha} B )$.
\end{proof}

\section{Sectional representation}

We first describe the Dauns-Hofmann Theorem on sectional representation for  C$^{\ast}$-algebras.  This originally appeared as~\cite[Corollary 8.13]{dauns_hofmann}, however we adopt a version that follows from~\cite[Theorem 3.3]{nilsen_bundles} for convenience.  For any C$^{\ast}$-algebra $A$ there is a canonically associated upper semicontinuous C$^{\ast}$-bundle (in the sense of~\cite[Section 1]{nilsen_bundles}) $\mathcal{A}$ over $\mathrm{Glimm}(A)$, if this space is locally compact, or $\beta \mathrm{Glimm}(A)$ otherwise. Under this representation,
\begin{enumerate}
\item[(i)] the  fibre algebras are $\ast$-isomorphic to the Glimm quotients $A / G_p$ of $A$, where we define $G_p = A$ for $p \in \beta \mathrm{Glimm}(A) \backslash \mathrm{Glimm}(A)$ if necessary, and
\item[(ii)] the map sending $a \mapsto \hat{a}$, where
\[
\hat{a} (p) = a + G_p \ \ (p \in \beta \mathrm{Glimm}(A))
\]
is a $\ast$-isomorphism of $A$ onto the C$^{\ast}$-algebra of  sections of the bundle $\mathcal{A}$ vanishing at infinity on $\mathrm{Glimm}(A)$.
\end{enumerate}

As a consequence of Theorem~\ref{t:homeo2}, we can see that for any C$^{\ast}$-algebras $A$ and $B$ the canonical bundle associated with $A \otimes_{\alpha} B$ has base space homeomorphic to $\left( \mathrm{Glimm}(A) \times \mathrm{Glimm}(B) , \tau_{cr} \right) $ (or its Stone-\v{C}ech compactification), and fibre algebras $\ast$-isomorphic to
\[
\frac{A \otimes_{\alpha} B }{\Delta (G_p , G_q)} \mbox{ for } (p,q) \in \mathrm{Glimm}(A) \times \mathrm{Glimm} (B),
\]
and zero otherwise.

Lee's Theorem~\cite[Theorem 4]{lee} asserts that the bundle $\mathcal{A}$ is in fact continuous (i.e. $A$ is $\ast$-isomorphic to a maximal full algebra of operator fields in the sense of Fell~\cite{fell}) if and only if the complete regularisation map $\rho_A$ is open (see also~\cite[Theorem 2.1]{arch_som_qs} and~\cite[Theorem 3.3]{nilsen_bundles}).  Thus in the case of a minimal tensor product of C$^{\ast}$-algebras, it is natural to ask whether  $\rho_{\alpha}$ being open is equivalent to $\rho_A$ and $\rho_B$ being open. It follows from~\cite[Lemma 2.2 and Theorem 2.3]{kaniuth} that this is indeed the case when $A \otimes_{\alpha} B$ satisfies property (F).  In this section we consider this question under more general hypotheses.

It is well known that $A \otimes_{\alpha} B$ satisfies property (F) if and only if $\Phi (I,J ) = \Delta (I,J)$ for all $(I,J) \in \mathrm{Id}' (A) \times \mathrm{Id}' (B)$; see~\cite[Theorem 5 (2)]{tomiyama} for example.  The assumption that $\Phi (G,H) = \Delta (G,H)$ for all Glimm ideals of $A$ and $B$ is weaker in general. For example, if $H$ is an infinite dimensional Hilbert space then $B(H) \otimes_{\alpha} B(H)$ does not satisfy property (F)~\cite[Corollary 7]{wasserman}.  However, $\mathrm{Glimm}(B(H)) $ is a one point space consisting of the zero ideal, and clearly $\Phi( \{ 0 \}, \{ 0 \} ) = \Delta ( \{ 0 \} , \{ 0 \} )$.
It appears to be unknown whether there exist C$^{\ast}$-algebras $A$ and $B$ and Glimm ideals $(G,H) \in \mathrm{Glimm}(A) \times \mathrm{Glimm}(B)$ such that $\Delta (G,H) \subsetneq \Phi (G,H)$.

The condition that $\Delta = \Phi$ on $\mathrm{Glimm}(A)$ is equivalent to requiring that the fibre algebras of the canonical  bundle associated with $A \otimes_{\alpha} B$ are given by the minimal tensor products of the corresponding fibres of the bundles of $A$ and $B$, that is
\begin{equation}
\label{e:fibre}
 \left\{ (A/G_p ) \otimes_{\alpha} (B/G_q) : (p,q) \in \mathrm{Glimm}(A) \times \mathrm{Glimm}(B) \right\},
\end{equation}
(with topology inherited from $\left( \mathrm{Glimm}(A) \times \mathrm{Glimm}(B) , \tau_{cr} \right)$ by Theorem~\ref{t:homeo2}).    Indeed one may always consider an element $c \in A \otimes_{\alpha} B$ as a cross-section of the fibred space given by~(\ref{e:fibre}) via $(p,q) \mapsto c + \Phi (G_p , G_q )$ for $(p,q) \in \mathrm{Glimm}(A) \times \mathrm{Glimm}(B)$~\cite[p. 136-137]{arch_cb}.  In the case that the bundles of $A$ and $B$ are continuous, it was shown in~\cite[Corollary 3.1]{arch_cb} that this representation of $A \otimes_{\alpha} B$ defines a continuous C$^{\ast}$-bundle over $\left( \mathrm{Glimm}(A) \times \mathrm{Glimm} (B) , \tau_{p} \right)$ in the obvious way  if and only if $\Delta = \Phi$ on $\mathrm{Glimm}(A) \times \mathrm{Glimm} (B)$.

Theorem~\ref{t:rho_alpha} below asserts that under the assumption $\Delta = \Phi $ on $\mathrm{Glimm}(A) \times \mathrm{Glimm} (B)$, the Dauns-Hofmann representation of $A \otimes_{\alpha} B$ defines a continuous C$^{\ast}$-bundle over $\mathrm{Glimm} (A \otimes_{\alpha} B )$ if and only if $A$ and $B$ define continuous C$^{\ast}$-bundles over $\mathrm{Glimm}(A)$ and $\mathrm{Glimm} (B)$ respectively.

\begin{lemma}
\label{l:ucapimphi}
Let $A$ and $B$ be C$^{\ast}$-algebras such that $\Delta (G,H) = \Phi (G,H)$ for all $(G,H) \in \mathrm{Glimm}(A) \times \mathrm{Glimm} (B)$, and let $\mathcal{U} \subseteq \mathrm{Prim}(A \otimes_{\alpha} B )$ be open.  Then 
\[ \rho_{\alpha} ( \mathcal{U} ) = \rho_{\alpha} \left( \mathcal{U} \cap \Phi ( \mathrm{Prim}(A) \times \mathrm{Prim}(B) ) \right) .\]
\end{lemma}
\begin{proof}
Let $m \in \rho_{\alpha} ( \mathcal{U} )$ and choose $M \in \mathcal{U}$ such that $\rho_{\alpha} (M) = m$.  Denote by $G_m = k \left( [ M ] \right)$ the corresponding Glimm ideal of $A \otimes_{\alpha} B$.  Then there exist $(p,q) \in \mathrm{Glimm}(A) \times \mathrm{Glimm}(B)$ and corresponding Glimm ideals $G_p$ and $G_q$ of $A$ and $B$ respectively with $G_m = \Phi (G_p,G_q) = \Delta (G_p , G_q )$ by Theorem~\ref{t:homeo2}. 

Now $M \in \mathrm{hull}(G_m) \cap \mathcal{U}$, which is a nonempty relatively open subset of $\mathrm{hull}(G_m)$.  By Proposition~\ref{p:phi}(iv), $\Phi  \left( \mathrm{hull} (G_p) \times \mathrm{hull} (G_q) \right)$ is dense in $\mathrm{hull} (G_m) = \rho_{\alpha}^{-1} ( \{ m \})$. Hence there exists $(P,Q) \in \mathrm{hull}(G_p) \times \mathrm{hull}(G_q)$ such that  $\Phi (P,Q) \in \mathrm{hull}(G_m) \cap\ \mathcal{U}$.  In particular $\Phi (P,Q) \in \mathcal{U} \cap  \Phi (\mathrm{Prim}(A) \times \mathrm{Prim}(B) ) $, and $\rho_{\alpha} \circ \Phi  (P,Q) = m$.  It follows that $\rho_{\alpha} ( \mathcal{U} ) \subseteq \rho_{\alpha} \left( \mathcal{U} \cap \Phi ( \mathrm{Prim}(A) \times \mathrm{Prim}(B) ) \right)$, and the reverse inclusion is trivial.

\end{proof}

\begin{theorem}
\label{t:rho_alpha}
Let $A$ and $B$ be C$^{\ast}$-algebras such that $\Delta (G,H) = \Phi (G,H)$ for all $(G,H) \in \mathrm{Glimm}(A) \times \mathrm{Glimm} (B)$.  Then the following are equivalent:
\begin{enumerate}
\item[(i)]  $\rho_{\alpha}$ is an open map with respect to $\tau_{cr}$ on $\mathrm{Glimm}(A \otimes_{\alpha} B )$,
\item[(ii)]  $\rho_{\alpha}$ is an open map with respect to $\tau_q$ on $\mathrm{Glimm}(A \otimes_{\alpha} B )$,
\item[(iii)] $\rho_A$ and $\rho_B$ are open maps with respect to $\tau_{cr}$ on $\mathrm{Glimm}(A)$ and $\mathrm{Glimm}(B)$ respectively,
\item[(iv)]  $\rho_A$ and $\rho_B$ are open maps with respect to $\tau_q$ on $\mathrm{Glimm}(A)$ and $\mathrm{Glimm}(B)$ respectively.
\end{enumerate}
\end{theorem}
\begin{proof}
Note that by Lemma~\ref{l:tauq}(i), (i) is equivalent to (ii), and (iii) is equivalent to (iv). We will show that (i) implies (iv) and that (iii) implies (i).

Suppose that (i) holds.  We first claim that $\rho_A \times \rho_B$ is open as a map into $\left( \mathrm{Glimm}(A) \times \mathrm{Glimm}(B) , \tau_{cr} \right)$.  Take an open subset $\mathcal{U} \subseteq \mathrm{Prim}(A) \times \mathrm{Prim} (B)$.  Then since the restriction of $\Phi$ to $\mathrm{Prim}(A) \times \mathrm{Prim}(B)$ is a homeomorphism onto its image by Proposition~\ref{p:phi}(ii), there is an open subset $\mathcal{U}' \subseteq \mathrm{Prim}(A \otimes_{\alpha} B)$ with $\mathcal{U}' \cap  \Phi ( \mathrm{Prim}(A) \times \mathrm{Prim}(B)) = \Phi ( \mathcal{U} )$.  By Lemma~\ref{l:ucapimphi} $\rho_{\alpha} ( \mathcal{U} ' ) = \rho_{\alpha} ( \Phi (\mathcal{U}))$.  Then by Theorem~\ref{t:homeo}
\[
(\rho_A \times \rho_B ) ( \mathcal{U} ) = (\psi \circ \rho_{\alpha} \circ \Phi ) ( \mathcal{U} ) = (\psi \circ \rho_{\alpha} ) ( \mathcal{U}'),
\]
which is $\tau_{cr}$-open since $\rho_{\alpha}$ is an open map and $\psi$ is a homeomorphism.

As in the proof of~\cite[p.351]{arch_som_qs}, $\tau_{cr}$ must agree with the quotient topology $\tilde{\tau}_q$ on $\mathrm{Glimm}(A) \times \mathrm{Glimm}(B)$ induced by the map $\rho_A \times \rho_B$. In particular $\rho_A \times \rho_B$ is a $\tilde{\tau}_q$-open map.   To see that $\rho_A$ is open, let $\mathcal{W} \subseteq \mathrm{Prim}(A)$ be open.  Then 
\[
(\rho_A \times \rho_B ) \left( \mathcal{W} \times \mathrm{Prim}(B) \right) = \rho_A ( \mathcal{W} ) \times \mathrm{Glimm}(B)
\]
is $\tilde{\tau}_q$-open, hence $(\rho_A \times \rho_B )^{-1} \left( \rho_A ( \mathcal{W} ) \times \mathrm{Glimm}(B) \right)$ is open.  By Lemma~\ref{l:prodcr} we have
\[
(\rho_A \times \rho_B )^{-1} \left( \rho_A ( \mathcal{W} ) \times \mathrm{Glimm}(B) \right) = \rho_A^{-1} \left( \rho_A ( \mathcal{W} ) \right) \times \mathrm{Prim}(B),
\]
so that in particular, $\rho_A^{-1} \left( \rho_A ( \mathcal{W} ) \right)$ is open.  It follows that $\rho_A$ is a $\tau_q$-open map.  A similar argument shows that $\rho_B$ is $\tau_q$-open, hence (i) implies (iv).

Assume that (iii) holds and take an open subset $\mathcal{U} \subseteq \mathrm{Prim}(A \otimes_{\alpha} B )$. Then $\mathcal{U} \cap \Phi ( \mathrm{Prim}(A) \times \mathrm{Prim}(B) )$ is a relatively open subset of $\Phi ( \mathrm{Prim}(A) \times \mathrm{Prim}(B) )$.  Again by Proposition~\ref{p:phi}(ii), there is an open subset $\mathcal{V} \subseteq \mathrm{Prim}(A) \times \mathrm{Prim}(B)$ such that $\Phi ( \mathcal{V} ) = \mathcal{U} \cap \Phi ( \mathrm{Prim}(A) \times \mathrm{Prim}(B) )$.  By Theorem~\ref{t:homeo} we then have
\[
(\psi \circ \rho_{\alpha} ) \left( \mathcal{U} \cap  \Phi( \mathrm{Prim}(A) \times \mathrm{Prim}(B) ) \right) = (\psi \circ \rho_{\alpha} \circ \Phi ) ( \mathcal{V} ) = (\rho_A \times \rho_B ) ( \mathcal{V} ),
\]
which is $\tau_{cr}$-open since $\rho_A \times \rho_B$ is a $\tau_{cr}$-open map by Lemma~\ref{l:tauq}(ii).  Together with Lemma~\ref{l:ucapimphi}, this shows that
\[
\rho_{\alpha} ( \mathcal{U} ) = \rho_{\alpha} \left( \mathcal{U} \cap \Phi ( \mathrm{Prim}(A) \times \mathrm{Prim}(B) ) \right) = \psi^{-1} \left( (\rho_A \times \rho_B )  ( \mathcal{V} ) \right),
\]
which is open since $\psi$ is a homeomorphism.  Hence $\rho_{\alpha}$ is an open map.

\end{proof}

Following a suggestion of R.J. Archbold, below we give a similar result to~\cite[Proposition 4.1]{arch_cb}.  Under the assumption that $A$ and $B$ each have at least one Glimm quotient containing a nonzero projection, we show in Proposition~\ref{p:proj}  that the implication (i)$\Rightarrow$(iii) of Theorem~\ref{t:rho_alpha} does not require $\Delta = \Phi$ on $\mathrm{Glimm}(A) \times \mathrm{Glimm}(B)$.  We establish as a corollary that under the same assumptions on $A$ and $B$, if $A \otimes_{\alpha} B$ is quasi-standard then $A$ and $B$ must be quasi-standard.

\begin{lemma}
\label{l:y0}
Let $X$ and $Y$ be topological spaces.  Then for any $y_0 \in Y$, the map sending $\rho_X (x) \mapsto \rho_{X \times Y} (x, y_0 )$ is a homeomorphic embedding of $\rho X$ into $\rho(X \times Y)$, with respect to the corresponding $\tau_{cr}$-topologies on each space.
\end{lemma}
\begin{proof}
By Lemma~\ref{l:prodcr} we may identify $\rho (X \times Y )$ with $\left( \rho X \times \rho Y ,\tau_{cr} \right)$ under the canonical mapping $\rho_{X \times Y} (x,y) \mapsto ( \rho_X (x) , \rho_Y (y) )$.  Clearly the map sending $\rho_X (x) \mapsto ( \rho_X (x) , \rho_Y (y_0 ) )$ is a homeomorphic embedding of $\rho X$ into $\rho X \times \rho Y $ with the product topology $\tau_p$. Thus we must show that the restrictions of the $\tau_p$ and $\tau_{cr}$ topologies to the subspace $\rho X \times \{ \rho_Y (y_0 ) \}$ are equal. Since $\tau_p \leq \tau_{cr}$ it will suffice to show that for any $\tau_{cr}$-open subset $\mathcal{O}$ of $\rho X \times \rho Y$ and $x_0 \in X$ such that $\left( \rho_X(x_0) , \rho_Y (y_0) \right) \in \mathcal{O}$, there is a cozero set neighbourhood $\mathcal{U}$ of $\rho_X (x_0)$ in $\rho X$ such that $\mathcal{U} \times \{ \rho_Y (y_0) \} \subseteq \mathcal{O}$.

Since $\mathcal{O}$ is $\tau_{cr}$-open there is $g \in C^b (X \times Y )$ such that $\mathrm{coz} (g^{\rho} )$ is a neighbourhood of $\left( \rho_X (x_0 ) , \rho_Y (y_0) \right)$ contained in $\mathcal{O}$.  Define $f \in C^b (X)$ via $f(x) = g (x , y_0 )$, then $f^{\rho} \in C^b ( \rho X )$ and $f^{\rho} \left( \rho_X (x) \right) = g^{\rho} \left( \rho_X (x) , \rho_Y (y_0 ) \right)$ for all $x \in X$.  In particular , $f^{\rho} ( \rho_X (x) ) = 0 $ if and only if $g^{\rho} \left( \rho_X (x) , \rho_Y (y_0 ) \right) = 0$, so that
\[
\mathrm{coz} (f^{\rho} ) \times \{ \rho_Y (y_0 ) \} = \mathrm{coz} (g^{\rho} ) \cap \left( \rho X  \times \{ \rho_Y (y_0) \} \right),
\]
as required.

\end{proof}

\begin{proposition}
\label{p:proj}
Suppose that  $A$ and $B$ are C$^{\ast}$-algebras such that the complete regularisation map $\rho_{\alpha} : \mathrm{Prim} (A \otimes_{\alpha} B ) \rightarrow \mathrm{Glimm}(A \otimes_{\alpha} B)$ is open. If there exists a point $q_0 \in \mathrm{Glimm}(B)$ (resp. $p_0 \in \mathrm{Glimm}(A)$) such that the quotient C$^{\ast}$-algebra $B/ G_{q_0}$ (resp. $A / G_{p_0}$) contains a nonzero projection, then $\rho_A$ (resp. $\rho_B$) is open.
\end{proposition}
\begin{proof}
Let $e \in B$ such that $e+G_{q_0}$ is a nonzero projection in $B/ G_{q_0}$.  Then the map $\Theta_p : A \rightarrow (A \otimes_{\alpha} B ) / \Delta (G_p , G_{q_0} )$ defined by
\[
\Theta_p (a) = a \otimes e + \Delta (G_p , G_{q_0})
\]
is a $\ast$-homomorphism for each $p \in \mathrm{Glimm}(A)$.  We first claim that $\mathrm{ker} \Theta_p = G_p$. 

Indeed, it is clear that if $a \in G_p$ then $a \otimes e \in \Delta (G_p , G_{q_0} )$, so that $G_p \subseteq \mathrm{ker} \Theta_p$.  Now choose a state $\lambda$ of $B$ vanishing on $G_{q_0}$ such that $\lambda (e) = 1$, and consider the associated left slice map $L_{\lambda} : A \otimes_{\alpha} B \rightarrow A$ defined on elementary tensors via $L_{\lambda} ( a \otimes b ) = \lambda (b) a$, and extended to $A \otimes_{\alpha} B$ by linearity and continuity.  Then since $L_{\lambda} ( A \odot G_{q_0} ) = \{ 0 \}$ and $L_{\lambda} ( G_p \odot B ) \subseteq G_p$, we have $L_{\lambda} ( \Delta(G_p, G_{q_0} ) ) \subseteq G_p$.  In particular, if $a \in \mathrm{ker} \Theta_p$ then $a \otimes e \in \Delta (G_p , G_{q_0} )$, so that
\[
L_{\lambda} ( a \otimes e ) = \lambda (e) a = a \in G_p,
\]
hence $\mathrm{ker} \Theta_p = G_p$. It follows that for any $a \in A$ and $p \in \mathrm{Glimm}(A)$, $\| a + G_p \| = \| \Theta_p (a) \|$.  

By~\cite[Theorem 2.1 (i) $\Rightarrow$ (ii)]{arch_som_qs} the function on $\mathrm{Glimm}(A \otimes_{\alpha} B )$ sending $ x \mapsto \| a \otimes e + G_x \|$ is continuous.  Since by Theorem~\ref{t:homeo2}, $\Delta: ( \mathrm{Glimm}(A) \times \mathrm{Glimm}(B) , \tau_{cr} ) \rightarrow \mathrm{Glimm}(A \otimes_{\alpha} B )$ is a homeomorphism, the map sending $(p,q) \mapsto \| a \otimes e + \Delta (G_p , G_q ) \|$ is $\tau_{cr}$-continuous on $\mathrm{Glimm}(A) \times \mathrm{Glimm}(B) $ for every $a \in A$.  

Finally, by Lemma~\ref{l:y0} the map $p \mapsto (p,q_0 )$ is a homeomorphic embedding of $\mathrm{Glimm}(A)$ into $\left( \mathrm{Glimm}(A) \times \mathrm{Glimm} (B) , \tau_{cr} \right)$. It follows that for each $a \in A$ the function $p \mapsto \| a + G_p \|$ agrees with the composition of continuous maps given by
\[
p \mapsto (p,q_0 ) \mapsto \Delta (G_p , G_{q_0} ) \mapsto \| a \otimes e + \Delta (G_p , G_{q_0} ) \|
\]
\[
\xymatrix{
\mathrm{Glimm}(A) \ \ar@{^{(}->}[r] & \left( \mathrm{Glimm}(A) \times \mathrm{Glimm}(B) , \tau_{cr} \right) \ar[r] & \mathrm{Glimm}(A \otimes_{\alpha} B ) \ar[r] &\mathbb{R}, \\
}
\]
hence is continuous.  By~\cite[Theorem 2.1 (ii) $ \Rightarrow$ (i)]{arch_som_qs}, this implies that $\rho_A$ is open.
\end{proof}

An ideal $I$ of a C$^{\ast}$-algebra $A$ is said to be \emph{primal} if whenever $n \geq 2$ and $J_1 , \ldots , J_n$ are ideals of $A$ with zero product then $J_i \subseteq I$ for at least one $i$ between 1 and $n$. We say that $A$ is \emph{quasi-standard} if the complete regularisation map $\rho_A$ is open, and every Glimm ideal of $A$ is primal.  There are many equivalent definitions of quasi-standard C$^{\ast}$-algebras, see~\cite[Theorems 3.3 and 3.4]{arch_som_qs} for example.

\begin{corollary}
\label{c:proj}
Suppose that $A$ and $B$ are C$^{\ast}$-algebras such that $A \otimes_{\alpha} B$ is quasi-standard.  If there exists a point $q_0 \in \mathrm{Glimm}(B)$ (resp. $p_0 \in \mathrm{Glimm}(A)$) such that the quotient C$^{\ast}$-algebra $B/ G_{q_0}$ (resp. $A / G_{p_0}$) contains a nonzero projection, then $A$ (resp. $B$) is quasi-standard.
\end{corollary}
\begin{proof}
Since $A \otimes_{\alpha} B$ is quasi-standard, the Glimm ideals of $A$ and $B$ are primal by~\cite[Lemma 4.1]{arch_cb}. The fact that $\rho_A$ and $\rho_B$ are open under the respective hypotheses then follows from Proposition~\ref{p:proj}. 
\end{proof}

We remark that if $Z(B) \neq \{ 0 \}$ then there is $q_0 \in \mathrm{Glimm}(B)$ such that $Z(B) \not\subseteq G_{q_0}$ (otherwise $Z(B) \subseteq  \bigcap \{ G_q :q \in \mathrm{Glimm}(B) \}  = \{ 0 \}$).  It then follows from~\cite[Proposition 2.2(ii)]{arch_som_mult} that $B / G_{q_0}$ is in fact unital, and in particular contains a nonzero projection.

This condition is not necessary for the assumptions of Proposition~\ref{p:proj} and Corollary~\ref{c:proj} to hold, as can be seen by taking $B = K(H)$ for a separable infinite dimensional Hilbert space $H$.  Then $\mathrm{Glimm}(B)$ consists of the zero ideal, so that $B$ is a Glimm quotient of itself.  We have $Z(B) = \{ 0 \}$, while $B$ contains all of the finite rank projections on $H$.

\section{Examples}

Our first example shows that the topologies $\tau_p$ and $\tau_{cr}$ can indeed differ for the complete regularisation of a product of primitive ideal spaces when condition (ii) of Theorem~\ref{t:ishii} fails.  We show that the primitive ideal space of the (separable) C$^{\ast}$-algebra $A$ of~\cite[III, Example 9.2]{dauns_hofmann} admits a point $P_0$ for which no cozero set neighbourhood of $P_0$ in $\mathrm{Prim}(A)$ has w-compact closure. Further we exhibit a cozero set neighbourhood $\mathcal{U}$ of $(P_0,P_0)$ in $\mathrm{Prim}(A) \times \mathrm{Prim}(A)$ which does not contain a product $\mathcal{V} \times \mathcal{W}$ of cozero sets $\mathcal{V}, \mathcal{W} \subseteq \mathrm{Prim}(A)$. Thus $\rho_A \times \rho_V ( \mathcal{U} )$ is a $\tau_{cr}$-open subset of $\mathrm{Glimm}(A) \times \mathrm{Glimm}(B)$ which is not $\tau_p$-open.  In particular we deduce that $\mathrm{Glimm}(A \otimes_{\alpha} A)$ is not homeomorphic to $\left( \mathrm{Glimm}(A) \times \mathrm{Glimm}(A), \tau_p \right)$.
\begin{example}
\label{e:notlocallycompact}
Let $H$ be a separable infinite dimensional Hilbert space, $ \{ e_n : n=1,2,\ldots \}$ a fixed orthonormal basis for $H$, and let $K(H)$ denote the compact operators on $H$.  The subset $D(H)$ of all compact operators that are diagonalisable with respect to the basis $\{e_n\}$ is a C$^{\ast}$-subalgebra of $K(H)$.  Let
\[
A = \{ F \in C([-1,1],K(H)) : F(t) \in D(H) \mbox{ for all } t \geq 0 \} .
\]
With pointwise operations and norm $ \| F \| = \mathrm{sup} \{ \| F(t) \| : t \in [-1,1] \}$, $A$ is a C$^{\ast}$-algebra.  Each element $F$ of $A$ is given by an infinite matrix of continuous functions $F_{i,j} : [-1,1] \rightarrow \mathbb{C}$, such that
\[
F(t)e_j = \sum \{F_{i,j}(t) e_i : i=1,2, \ldots \}.
\]
The set $\mathrm{Prim}(A)$ consists of the ideals
\begin{eqnarray*}
P(t) &=& \{ F \in A : F(t)=0 \} \mbox{ for } t <0 \\
P(t,n) &=& \{ F \in A : F_{n,n} (t) =0 \} \mbox{ for } t \geq 0, n \in \mathbb{N} \\
\end{eqnarray*}

The topology on $\mathrm{Prim} (A)$ is as follows:
\begin{itemize}
\item For $t<0$, a neighbourhood basis for $P(t)$ is given by the sets
\[
N(t, \varepsilon ) : = \{ P(q) : |q-t|< \varepsilon , q <0 \}
\]
\item For $t=0$ and $n \in \mathbb{N}$, $P(0,n)$ has a neighbourhood basis of sets
\[
M(0,n, \varepsilon ) : = \{ P(q) : - \varepsilon < q < 0 \} \cup \{ P(q,n) : 0<q< \varepsilon \}
\]
\item For $t>0$ and $n \in \mathbb{N}$, a neighbourhood basis for  $P(t, n )$ is given by the sets
\[
M(t,n, \varepsilon ) : = \{ P(q,n): q>0, |q-t| < \varepsilon \} 
\]
\end{itemize}

For each $n$ let $I_n = [ 0,1 ] \times \{ n \}$, and for $t \in [0,1]$ let $t^{(n)} = (t,n) \in I_n$. Then
\[
\mathrm{Prim} (A) \equiv  [-1,0)  \cup (  \bigcup_{n=1}^{\infty} I_n ),
\]
where each point $0^{(n)}$ has a neighbourhood basis of intervals $(-\varepsilon,0)  \cup [0^{(n)}, \varepsilon^{(n)} ) $.  It is easy to see that for a pair of points $0^{(n)}$ and $0^{(m)}$, with $n \neq m$, any open neighbourhood of $0^{(n)}$ will intersect every open neighbourhood of $0^{(m)}$ in the subset $ [-1,0) $.

The complete regularisation map $\rho : \mathrm{Prim} (A) \rightarrow \mathrm{Glimm}(A)$ fixes the sets $[-1,0)$ and $ (0^{(n)},1^{(n)}]$ for all $n$.  The set $\{ 0^{(n)} : n \in \mathbb{N} \}$ is mapped to a single point $0^{(0)}$.  Moreover, $0^{(0)}$ does not have a compact neighbourhood in $\mathrm{Glimm}(A)$. For a proof of these facts, see~\cite[Examples 3.4 and 9.2]{dauns_hofmann}.  Hence
\[
\mathrm{Glimm}(A) =  [-1,0]  \cup \bigcup_{n=1}^{\infty} (0^{(n)},1^{(n)}] ,
\]
where a neighbourhood basis of $0^{(n)}$ is given by the sets $(-\delta_0,0] \cup \bigcup_{n=1}^{\infty} (0^{(n)}, \delta_n^{(n)})$ where $\delta_j > 0$ for all $j \geq 0$.  $\mathrm{Glimm}(A)$ is homeomorphic to the subset of $\mathbb{C}$ given by
\[
[-1,0] \cup \{ r e^{i \theta } : 0 < r \leq 1, \theta=\tfrac{n-1}{2n}  \pi \mbox{ for } n \geq 1 \}.
\]  

We will show that no $0^{(n)}$ has a cozero set neighbourhood in $\mathrm{Prim} (A)$ with w-compact closure.  Indeed any such neighbourhood is of the form $\rho^{-1} (V)$ where $V$ is a cozero set neighbourhood of $0$ in $\mathrm{Glimm} (A)$, and is hence a neighbourhood of $0^{(j)}$ for all $j \in \mathbb{N}$. Hence for every $j \in \mathbb{N} \cup \{ 0 \}$ there is $\varepsilon_j > 0 $ such that
\[
U : = ( - \varepsilon_0 , 0 ) \cup \bigcup_{n=1}^{\infty} [0^{(n)}, \varepsilon_n^{(n)} ) \subseteq \rho^{-1} (V).
\]
Note that
\begin{eqnarray*}
\overline{U} &=& [ - \varepsilon_0, 0 ) \cup \bigcup_{n=1}^{\infty} [ 0^{(n)}, \varepsilon_n^{(n)} ], \\
\rho (U) & = & ( -\varepsilon_0, 0 ] \cup \bigcup_{n=1}^{\infty} (0^{(n)} , \varepsilon_n^{(n)} ) ,
\end{eqnarray*}
and that $\rho (U)$ is clearly a cozero set of $\mathrm{Glimm} (A)$, hence $U$ is a cozero set of $\mathrm{Prim} (A)$.  

We claim that it suffices to consider a cozero set neighbourhood of $0^{(n)}$ of the form $U$. Indeed, if $\overline{\rho^{-1}(V)}$ were w-compact, then $\overline{U}$, being the closure of a cozero set of a w-compact space, would also be w-compact by~\cite[Proposition 3.8]{mor_ish}.  Therefore if $\overline{U}$ is not w-compact, then no cozero set neighbourhood $\rho^{-1}(V)$ of $0^{(n)}$ can be w-compact.

For each $n$ let $U_n = [ - \varepsilon_0, 0 ) \cup [ 0^{(n)} , \varepsilon_n^{(n)} ] $.  Then the collection $\{ U_n \}$ is an open cover of $\overline{U}$.

If we take any finite subcollection, w.l.o.g $U_1, \ldots , U_n$, then $\mathrm{cl}_{\tau_{\bar{U}}} ( U_1 \cup \ldots \cup U_n )$ consists of all points $x \in \overline{U}$ such that any cozero set neighbourhood of $x$ intersects $U_1 \cup \ldots \cup U_n$.  If $m>n$ and $y^{(m)} \in (0^{(m)} , \varepsilon_m^{(m)} ]$ then we can choose $g \in C^b ( (0^{(m)} , \varepsilon_m^{(m)}] )$  with $g(y^{(m)}) = 1$, vanishing off a compact neighbourhood of $y^{(m)}$ in $(0 , \varepsilon_m ]$.  Extending $g$ to be zero elsewhere on $\overline{U}$ gives a cozero set neighbourhood of $y^{(m)}$ disjoint from $U_1 \cup \ldots \cup U_n$. It follows that
\[
\mathrm{cl}_{\tau_{\bar{U}}} ( U_1 \cup \ldots \cup U_n ) = U_1 \cup \ldots \cup U_n \cup \{ 0^{(m)} : m > n \},
\]
a proper subset of $\overline{U}$.  So $0^{(n)}$ does not have a cozero set neighbourhood with w-compact closure for any $n$.

We now show that the product topology $\tau_p$ on $\mathrm{Glimm} (A) \times \mathrm{Glimm} (A)$ is strictly weaker than $\tau_{cr}$.  Note first that
\[
\mathrm{Prim} (A) \times \mathrm{Prim} (A) = [-1,0) \times [-1,0) \cup \left( [-1,0) \times \bigcup_{n=1}^{\infty} I_n \right) \cup \left( \bigcup_{m=1}^{\infty} I_m \times [-1,0) \right) \cup \left( \bigcup_{m,n=1}^{\infty} I_m \times I_n \right).
\]
The neighbourhood basis of the following types of points will be of interest:
\begin{itemize}
\item $(0^{(m)},0^{(n)} )$: sets of the form
\[
\left( ( - \delta , 0 ) \cup [ 0^{(m)} , \delta^{(m)} ) \right) \times \left( ( - \varepsilon , 0 ) \cup [0^{(n)}, \varepsilon^{(n)} ) \right)
\]
where $\delta, \varepsilon >0 $.
\item $(x , y ) \in \{0^{(m)} \} \times ( 0^{(n)}, 1^{(n)} ] $: neighbourhoods of the form
\[
\left( ( - \delta, 0 ) \cup [ 0^{(n)} , \delta ) \right) \times \left( ( (y-\varepsilon)^{(n)} , ( y+ \varepsilon )^{(n)} ) \right)
\]
where $0< \delta < 1, 0< \varepsilon < | y |$.
\item $(x,y) \in (0^{(m)},1^{(m)}] \times \{ 0^{(n)} \}$: neighbourhoods of the form
\[
 \left( (x- \delta)^{(m)} , (x+\delta)^{(m)} \right)  \times \left( ( - \varepsilon , 0 ) \cup [ 0^{(n)}, \varepsilon^{(n)} ) \right)
\]
\end{itemize}
For each $m,n \in \mathbb{N}$ define $f_{m,n} : I_m \times I_n \rightarrow [0,1]$ via
\[
f_{m,n} (x,y) = \mathrm{max} (1-mnx, 1-mny,0 ).
\]
Then $f_{m,n} (x,y) > 0$ when $x< \frac{1}{mn}$ or $y< \frac{1}{mn}$, i.e. $\mathrm{coz}( f_{m,n} ) = [0^{(m)}, \left( \frac{1}{mn} \right)^{(m)} ) \times [0^{(n)} , \left( \frac{1}{mn} \right)^{(n)} )$. Now define  $f : \mathrm{Prim} (A) \times \mathrm{Prim} (A) \rightarrow [0,1]$ via
\[
f(x,y) = \left\{ \begin{array}{cl}
		 f_{m,n} (x,y) & \mbox{ if } (x,y) \in I_m \times I_n \\
		1 & \mbox{ otherwise} \\
		
		\end{array} \right. \]
Then $f$ is continuous since $f_{m,n} ( 0^{(m)} , 0^{(n)} ) = f_{m,n} ( 0^{(m)} , y^{(n)} ) = f_{m,n} (x^{(m)},0^{(n)} ) = 1$, and by the neighbourhood bases of these points constructed above.  Moreover, the cozero set of $f$ is
\[
[-1,0) \times [-1,0) \cup \left( [-1,0) \times \bigcup_{n=1}^{\infty} I_n \right) \cup \left( \bigcup_{m=1}^{\infty} I_m \times [-1,0) \right) \cup \left( \bigcup_{m,n=1}^{\infty} [ 0^{(m)} , \left( \tfrac{1}{mn} \right)^{(m)} ) \times [0^{(n)} , \left( \tfrac{1}{mn} \right)^{(n)} ) \right)
\]

We show that $\mathrm{coz} (f)$ is not a union of cozero set rectangles.  Indeed, for $i,j \in \mathbb{N}$, suppose $U \times V$ were a cozero set neighbourhood of $(0^{(i)},0^{(j)})$ contained in $\mathrm{coz} (f)$.  As before we may assume w.l.o.g. that
\begin{eqnarray*}
U &=& ( - \delta_0,0 ) \cup \bigcup_{m=1}^{\infty} [ 0^{(m)} , \delta_m^{(m)} ) \\
V &=& ( - \varepsilon_0 , 0 ) \cup \bigcup_{n=1}^{\infty} [0^{(n)} , \varepsilon_n^{(n)} ),
\end{eqnarray*}
where $\delta_j,\epsilon_j>0$ for all $j \geq 0$.  If $U \times V \subseteq \mathrm{coz} (f)$ then in particular it must be true that
\[
\textstyle [0^{(m)} , \delta_m^{(m)} ) \times [0^{(n)}, \varepsilon_n^{(n)} ) \subseteq [0^{(m)} , \left( \frac{1}{mn} \right)^{(m)} ) \times [0^{(n)} , \left( \frac{1}{mn} \right)^{(n)} )
\]
for all $m,n \geq 1$.  In other words, $\delta_m \leq \frac{1}{mn}$ for all $n \geq 1$ and $\varepsilon_n \leq \frac{1}{mn}$ for all $m \geq 1$.  But then $\delta_m = \varepsilon_n = 0 $ for all $m,n \geq 1$.
\qed

\end{example}

In what follows we denote by $\omega_0$ the first infinite ordinal and by $\omega_1$ the first uncountable ordinal.  For $i=0,1$ we let $[0, \omega_i )$ be the space of all ordinals $\gamma < \omega_i$ and $[0,\omega_i ] = [0, \omega_i+1)$.  These spaces will be considered with the order topology, with basic open sets given by
\[
( \alpha , \beta ) : = \{ \gamma \in [0, \omega_i ] : \alpha < \gamma < \beta \},
\]
where $\alpha, \beta \in [0, \omega_i ]$ for $i=0,1$.  A useful property of the space $[0, \omega_1)$ is that $\beta ( [0 , \omega_1) ) = [0 , \omega_1]$~\cite[5.13]{gill_jer}.

It follows from~\cite[5.11(c) and 5.12(c)]{gill_jer} that the space $[0, \omega_1)$ is a non-compact pseudocompact space.  On the other hand $[0, \omega_0 )$ is homeomorphic to $\mathbb{N}$, which being infinite and discrete cannot be pseudocompact.

Our second example is a nontrivial application of Theorem~\ref{p:zma}. First we describe the C$^{\ast}$-algebra $A$ of~\cite[Appendix]{lazar_quot}, which has the property that $\mathrm{Glimm}(A)$ is pseudocompact but non-compact.  We then construct a (non-unital) $\sigma$-unital C$^{\ast}$-algebra $B$ with $\mathrm{Glimm}(B)$ compact, such that $M(A) \otimes_{\alpha} M(B) \neq M( A \otimes_{\alpha} B )$, while $ZM(A) \otimes ZM(B) = ZM(A \otimes B)$.   
\begin{example}
\label{e:zma}
Let $Y = [0, \omega_1] \times [0 , \omega_0 ] \backslash \{(\omega_1 , \omega_0 ) \}$ and denote by $S = \{ \omega_1 \} \times [ 0 , \omega_0 )$ and $T = [0 , \omega_1 ) \times \{ \omega_0 \}$.  Define a new space $X = Y \cup \{ y \}$, where $y \not\in Y$ with topology such that $Y$ is embedded homeomorphically into $X$, and $\{y\}$ is an open set whose closure is $S \cup \{ y \}$.  

Let $C = C_0 (Y)$, $D = C_0 (S)$ and let $\pi_1 : C \rightarrow D$ be the restriction map.  Let $H$ be an infinite dimensional separable Hilbert space and $ \{ p_n \}$ a sequence of infinite dimensional mutually orthogonal projections on $H$.  Define an injective $\ast$-homomorphism $ \lambda : D \rightarrow B(H)$ via $\lambda (f) = \sum_{n=1}^{\infty} f ( \omega_1 , n ) p_n$, and note that $\lambda (D) \cap K(H)  =  \{ 0 \}$.

Set $E = \lambda (D) + K(H)$ and let $\pi_2 : E \rightarrow D$ be the quotient map.  Let $A = \{ (c,e) \in C \oplus E : \pi_1 (c) = \pi_2 (e) \}$.  Then $\mathrm{Prim}(A)$ is homeomorphic to $X$.

The complete regularisation map $\rho_{A}$ maps $Y \backslash S$ to itself, and $S \cup \{ y \} $ to a single point which we will denote by $z$.  Thus $\mathrm{Glimm}(A) = \left( Y \backslash S \right) \cup \{ z \}$, where a neighbourhood basis of $z$ is given by the collection of sets of the form
\[
 \left( \bigcup_{n=1}^{\omega_0} (\alpha_n , \omega_1 ) \right) \cup \{ z \},
\] 
for some ordinals $0 < \alpha_n < \omega_1$ for all $1 \leq n \leq \omega_0$.

Note that $Y \backslash S = [0 , \omega_1 ) \times [ 0 , \omega_0 ]$, being the product of a pseudocompact space and a compact space, is necessarily pseudocompact.  It follows that $\mathrm{Glimm}(A)$ is pseudocompact.

Consider the C$^{\ast}$-algebra $B$ of sequences $(T_n) \in B(H)$ such that $T_n \rightarrow T_{\infty} \in K(H)$, with pointwise operations and supremum norm.  Then $\mathrm{Prim}(B)$ consists of the ideals
\[
P_{n_0} = \{ (T_n) : T_{n_0} = 0 \} , K_{n_0} = \{ (T_n) : T_{n_0} \in K(H) \}
\]
for $n_0 \in \mathbb{N}$, and $P_{\infty} = \{ (T_n) : T_{\infty} = 0 \}$.  The $\approx$-equivalence classes in $\mathrm{Prim}(B)$ then consist of pairs $\{ P_{n_0} , K_{n_0} \}$ for $n_0 \in \mathbb{N}$, and $\{ P_{\infty} \}$.  As in the proof of~\cite[Proposition 3.6]{arch_som_qs}, the complete regularisation map $\rho_B : \mathrm{Prim}(B) \rightarrow \mathrm{Glimm}(B)$ is open and $\mathrm{Glimm}(B )$ is homeomorphic to $\mathbb{N} \cup \infty$, with $G_q = P_q$ for all $1 \leq q \leq \infty$.

We claim that $B$ is a $\sigma$-unital C$^{\ast}$-algebra.  Fix an orthonormal basis $\{ e_{m} : m \in \mathbb{N} \}$ of $H$.  For each $n \in \mathbb{N}$ let $1_n$ denote the projection onto the $n$-dimensional subspace of $H$ spanned by $e_1, \ldots , e_n$.  Then $\{ 1_n : n \in \mathbb{N} \}$ is an increasing approximate identity for $K(H)$.

Now define sequences $E^{(n)} = ( E^{(n)}_m )$ in $B$ via 
\[
E^{(n)}_m = \left\{ 
\begin{array}{rcl}
 1 & \mbox{ if } & m \leq n \\
 1_n & \mbox{ if } & m>n \\ 
\end{array} \right. \]
for each $n \in \mathbb{N} $.  We will show that the sequence $ \{ E^{(n)}  \}_{n}$ is an approximate identity for $B$.  Take $T = ( T_m ) \in B$ and let $\varepsilon > 0$ be given. Note that for each $n$, 
\[
\| T - E^{(n)} T \| = \sup_{m \geq 1} \| T_m - E^{(n)}_m T_m \| = \sup_{m > n} \| T_m - 1_n T_m \|,
\]
by definition of $E^{(n)}$. Then
\begin{itemize}
\item there exists $m_0 \geq 1$ such that $ \| T_m - T_{\infty} \| < \frac{\varepsilon}{4} $ whenever $m \geq m_0$, and
\item there exists $n_1 \geq 1$ such that $ \| T_{\infty} - 1_n T_{\infty} \| < \frac{\varepsilon}{4}$ whenever $n \geq n_1$ (since $T_{\infty} \in K(H)$ ).
\end{itemize}

Set $n_0 = \max (m_0 , n_1 )$.  Then if $n \geq n_0$ and $m>n$ we have
\begin{eqnarray*}
\| T_m - 1_n T_m \| & \leq & \| T_m - T_{\infty} \| + \| T_{\infty} - 1_n T_m \| \\
& \leq & \| T_m - T_{\infty} \| + \| T_{\infty} - 1_n T_{\infty} \| + \| 1_n T_{\infty} - 1_n T_m \|  \\
& \leq & \| T_m - T_{\infty} \| + \| T_{\infty} - 1_n T_{\infty} \| + \| T_{\infty} - T_m \| \\
& < & \frac{\varepsilon}{4} + \frac{\varepsilon}{4} + \frac{ \varepsilon}{4} = \frac{3 \varepsilon}{4} .
\end{eqnarray*}
In particular, for all $n \geq n_0$ we have
\[
\| T - E^{(n)} T \| = \sup_{m > n} \| T_m - 1_n T_m \| \leq \frac{3 \varepsilon}{4} < \varepsilon,
\]
so that $ \{ E^{(n)} \}_n$ is a countable approximate identity for $B$.

Now take the tensor product $A \otimes_{\alpha} B$.  Since $B$ is $\sigma$-unital, the inclusion $M(A) \otimes_{\alpha} M(B) \subseteq M ( A \otimes_{\alpha} B )$ is strict by~\cite[Theorem 3.8]{apt}.  Since $\rho_B$ is open, Proposition~\ref{p:wcpct} (ii) shows  that $\tau_{cr} = \tau_p$ on $\mathrm{Glimm}(A) \times \mathrm{Glimm}(B)$.  Hence by Theorem~\ref{t:homeo} $\mathrm{Glimm}(A \otimes_{\alpha} B)$ is homeomorphic to $\mathrm{Glimm}(A) \times \mathrm{Glimm}(B)$ with the product topology.  Moreover, as a product of a pseudocompact space and a compact space, $\mathrm{Glimm}(A) \times \mathrm{Glimm}(B)$ is pseudocompact~\cite[Proposition 8.21]{walker}.  It then follows from Theorem~\ref{p:zma} that $ZM(A) \otimes ZM(B) = ZM(A \otimes_{\alpha} B )$.

\qed
\end{example}

\textbf{Acknowledgement: } The author thanks R.J. Archbold for suggesting the topic initially and for detailed comments on an earlier draft of the paper.
\bibliography{refs2}
\bibliographystyle{rt}
\end{document}